\documentclass[pdflatex]{sn-jnl}

\catcode`\|=12\relax

\jyear{2022}
\newlength{\snfigurewidth}
\newlength{\snfigureheightmax}
\usepackage[utf8]{inputenc}
\usepackage[T1]{fontenc}
\usepackage{lmodern}

\usepackage{amsmath,amssymb,amsthm}
\usepackage{paralist}

\usepackage{subcaption}
\usepackage{todonotes}

\usepackage{xparse}
\usepackage{hyperref}
\usepackage[
    capitalise,
    noabbrev,
  ]{cleveref}

\DeclareDocumentCommand{\eqref}{m}{\labelcref{#1}}
\usepackage[shortlabels]{enumitem}

\crefname{setting}{Setting}{Settings}
\crefname{setting-s}{Setting}{Settings}
\crefname{assumption-t}{Assumption}{Assumptions}

\DeclareDocumentCommand{\InputData}{mO{}}{
  \def\DataPrefix{#2}
  \input{#1}}

\usepackage{pgffor}

\usepackage{tikz}
\usetikzlibrary{calc}

\usepackage{pgfplotstable}
\usepackage{pgfplots}
\pgfplotsset{compat=1.12}

\usepgfplotslibrary{external}

\pgfplotsset{every axis/.append style={
    no markers
}}
\pgfplotsset{every axis plot/.append style={
    line width=0.6pt
}}
\pgfplotsset{every axis legend/.append style={
    at={(1.02,1)},
    anchor=north west
}}

\definecolor{xgfs-normal6-1}{RGB}{64, 83, 211}
\definecolor{xgfs-normal6-2}{RGB}{221, 179, 16}
\definecolor{xgfs-normal6-3}{RGB}{181, 29, 20}
\definecolor{xgfs-normal6-4}{RGB}{0, 190, 255}
\definecolor{xgfs-normal6-5}{RGB}{251, 73, 176}
\definecolor{xgfs-normal6-6}{RGB}{0, 178, 93}

\pgfplotscreateplotcyclelist{xgfs-normal6}{
    {xgfs-normal6-1},
    {xgfs-normal6-2},
    {xgfs-normal6-3},
    {xgfs-normal6-4},
    {xgfs-normal6-5},
    {xgfs-normal6-6},
}
\pgfplotsset{every axis/.append style={cycle list name=xgfs-normal6}}

\DeclareDocumentCommand{\N}{}{\mathbb{N}}
\DeclareDocumentCommand{\Z}{}{\mathbb{Z}}
\DeclareDocumentCommand{\R}{}{\mathbb{R}}

\DeclareDocumentCommand{\<}{}{\langle}
\DeclareDocumentCommand{\>}{}{\rangle}

\DeclareDocumentCommand{\Hzdiv}{}{H_0^{\mathrm{div}}}

\DeclareDocumentCommand{\op}{m}{\operatorname{#1}}
\DeclareDocumentCommand{\Div}{}{\operatorname{div}}

\DeclareDocumentCommand{\supp}{}{\operatorname*{supp}}
\DeclareDocumentCommand{\argmin}{}{\operatorname*{arg\,min}}

\newtheorem{theorem}{Theorem}[section]
\newtheorem{definition}[theorem]{Definition}
\newtheorem{lemma}[theorem]{Lemma}
\newtheorem{proposition}[theorem]{Proposition}
\newtheorem{corollary}[theorem]{Corollary}
\newtheorem{remark}[theorem]{Remark}

\DeclareDocumentCommand{\di}{O{x}}{\mathord{\,\mathrm{d}#1}}

\ExplSyntaxOn
\definecolor{algcomment}{HTML}{444444}
\definecolor{algcommand}{HTML}{000055}

\DeclareDocumentCommand { \Return } { m } { \State { \mycourse_algcommand_style: return} ~ #1 }

\ExplSyntaxOff

\DeclareDocumentCommand{\Code}{m}{%
  \begingroup
  \ttfamily
  \begingroup\lccode`~=`/\lowercase{\endgroup\def~}{/\discretionary{}{}{}}%
  \begingroup\lccode`~=`[\lowercase{\endgroup\def~}{[\discretionary{}{}{}}%
  \begingroup\lccode`~=`.\lowercase{\endgroup\def~}{.\discretionary{}{}{}}%
  \catcode`/=\active\catcode`[=\active\catcode`.=\active\catcode`_=\active
  \scantokens{#1\noexpand}%
  \endgroup%
}

\DeclareDocumentCommand{\op}{m}{\operatorname{#1}}
\DeclareDocumentCommand{\Hzdiv}{}{H_0^{\mathrm{div}}}

\makeatletter
\renewcommand{\todo}[2][]{\tikzexternaldisable\@todo[#1]{#2}\tikzexternalenable}
\newcommand{\tododone}[2][]{\tikzexternaldisable\@todo[color=green!80!black,#1]{#2}\tikzexternalenable}
\makeatother

\usepackage[11pt]{moresize}
\let\oldfrac\frac
\renewcommand{\frac}[2]{
  \mathchoice
    {\tfrac{#1}{#2}}
    {\oldfrac{#1}{#2}}
    {\oldfrac{#1}{#2}}
    {\oldfrac{#1}{#2}}
}

\NewDocumentCommand{\LambdaStar}{}{\Lambda}
\NewDocumentCommand{\LambdaOp}{}{\Lambda^*}
\NewDocumentCommand{\NormLambda}{}{\|\Lambda\|}
\def\D{\mathcal{D}}

\begin{document}

\title[A General Decomposition Method for a Convex Problem]{A General Decomposition Method for a Convex Problem Related to Total Variation Minimization}
\author[2]{\fnm{Stephan} \sur{Hilb}}
\author*[1]{\fnm{Andreas} \sur{Langer}}\email{andreas.langer@math.lth.se}

\affil*[1]{%
  \orgdiv{Department for Mathematical Sciences}, %
  \orgname{Lund University}, %
  \orgaddress{%
    \street{Box~117}, %
    \postcode{221 00}, %
    \city{Lund}, %
  \country{Sweden}}%
}
\affil[2]{%
  \orgaddress{%
    \city{Stuttgart}, %
  \country{Germany}}%
}

\abstract{
We consider sequential and parallel decomposition methods for a dual problem of a general total variation minimization problem with applications in several image processing tasks, like image inpainting, estimation of optical flow and reconstruction of missing wavelet coefficients. The convergence of these methods to a solution of the global problem is analysed in a Hilbert space setting and a convergence rate is provided. Thereby, these convergence result hold not only for exact local minimization but also if the subproblems are just solved approximately. As a concrete example of an approximate local solution process a surrogate technique is presented and analysed. Further, the obtained convergence rate is compared with related results in the literature and shown to be in agreement with or even improve upon them. Numerical experiments are presented to support the theoretical findings and to show the performance of the proposed decomposition algorithms in image inpainting, optical flow estimation and wavelet inpainting tasks.
}



\maketitle


\section{Introduction}
The dimensionality of images has been tremendously increased in recent years due to the improvement of hardware. In order to further post-process such large-scale data in a distributed parallel or memory-constrained setting, decomposition methods may be used, which split the original problem into a sequence of smaller subproblems that can be solved independently of each other while still approaching the original solution by means of an iterative algorithm. One
particular approach are domain decomposition algorithms \cite{DolJolNat, QuaVal, TosWid} which subdivide the
problem domain. Typical examples of post-processing images include the removal of noise (denoising), the completion of missing data (inpainting) and the analysis of the data, as the computation of the optical flow in image sequences. In such applications one is usually interested in solutions in which edges are preserved. The total variation (TV) is well-know to promote discontinuities and hence is widely used in image processing tasks. Thereby one may consider the following regularized TV-model, cf. \cite{HinKun}, 
\begin{equation} \label{eq:model-motivation}
  \inf_{u \in L^2(\Omega)^m \cap BV(\Omega)^m}
  \tfrac{1}{2} \|Tu - g\|_{L^2(\Omega)}^2  +
  \tfrac{\beta}{2}\|u\|_{L^2(\Omega)}^2 + \lambda \int_\Omega |Du|,
\end{equation}
where $\Omega \subset \R^d$, $d \in \N$, is an open, bounded and simply
connected domain with Lipschitz boundary, $g\in L^2(\Omega)$ describes the observed data, $T: L^2(\Omega)^m \to L^2(\Omega)$ with $m\in\N$ is a linear bounded operator, $\beta\geq 0$, $\lambda >0$, and $\int_\Omega |Du|$ denotes the total variation of $u$ in $\Omega$ defined by
\begin{equation} \label{eq:tv}
  \begin{aligned}
    \int_\Omega |Du|:= \sup &\Big\{\int_{\Omega} u \cdot \operatorname{div}\vec v \di[x] \;:\;
  \vec v\in (C_0^\infty(\Omega))^{d\times m}, \\
                            &\qquad |\vec v(x)|_F \leq 1 \text{
for almost every (f.a.e.) $x\in\Omega$} \Big\},
  \end{aligned}
\end{equation}
with $|\cdot|_F:\R^{d\times m} \mapsto \R$ being the Frobenius norm, cf. \cite{ours2022semismooth}. We recall that $BV(\Omega)^m$, i.e., the space of functions with bounded variation, equipped with the norm $\|\cdot\|_{BV}:= \|\cdot\|_{L^1(\Omega)} + \op{TV}(\cdot)$ is a Banach space \cite[Theorem 10.1.1]{AttButMic}. Note that $m\in\N$ describes the number of output channels, e.g., for grey-scale images we set $m=1$ while for motion fields we have $m=d$. 

The crucial difficulty of deriving decomposition methods for total variation minimization problems lies in the fact that the total variation is non-differentiable and non-additive with respect to a disjoint splitting of the domain $\Omega$. In fact, let $\Omega_1$ and $\Omega_2$ be a disjoint decomposition of $\Omega$, then we have the following splitting property, cf. \cite[Theorem 3.84]{AmbFusPal},
\begin{equation}\label{Eq:TVsplitting}
  \begin{aligned}
\int_\Omega |D(u_{\mid_{\Omega_1}}+u_{\mid_{\Omega_2}})| &=
\int_{\Omega_1}|D(u_{\mid_{\Omega_1}})|+
\int_{\Omega_2}|D(u_{\mid_{\Omega_2}})| \\
&\qquad+ \int_{\partial \Omega_1 \cap \partial \Omega_2} |u_{\mid_{\Omega_1}}^+-u_{\mid_{\Omega_2}}^-|\text{ d} \mathcal{H}^{d-1}(x),
  \end{aligned}
\end{equation}
where $\mathcal{H}^{d}$ denotes the Hausdorff measure of dimension $d$ and the symbols $u^+$ and $u^-$ the ``interior'' and ``exterior'' trace of $u$ on $\partial \Omega_1 \cap \partial \Omega_2$ respectively. That is the total variation of a function of the whole domain equals the sum of the total variation on the subdomains plus the size of the possible jumps at the interface. Exactly these jumps at the interfaces are important as we want to preserve crossing discontinuities and the correct matching where the solution is continuous. A failure of a decomposition method for total variation minimization has been reported in \cite{ForKimLanSch} with respect to a wavelet space decomposition. There a condition is derived which allows to check for global optimality of a limit point generated by the decomposition method. Although the method seems to work fine in practice, a counterexample showed in \cite{ForKimLanSch} that this condition does not hold in general. Nevertheless, this condition may be utilized in order to check aposteriori whether the splitting method found a good numerical approximation. 
First domain decomposition methods for total variation minimization are presented in \cite{ForLanSch2010, ForSch, LanOshSch}.
Although their convergence and monotonic decay of the energy is theoretically ensured, the convergence to the solution of the global problem cannot be guaranteed in general, as counterexamples in \cite{Lan2021, LeeNam} illustrate. However, in \cite{HinLan2013, HinLan2014} an estimate of the distance of the numerical solution generated by such a decomposition method to the true minimizer of the original problem is derived. Utilizing this estimate demonstrated in \cite{HinLan2013, HinLan2014} that the splitting methods work in practice quite well for total variation minimization, as they indeed generate sequences for which this estimate indicates convergence to the global minimizer.

To overcome the difficulties due to the minimization of a non-smooth and non-additive objective in \eqref{eq:model-motivation} a predual problem of \eqref{eq:model-motivation}, as in  \cite{ChaTaiWanYan, HinLan2015_1} for the case $T=I$, $\beta=0$, and $m=1$, may be considered. In fact, a predual formulation of \eqref{eq:model-motivation} can be derived which involves constrained minimization of a smooth functional \cite{ours2022semismooth}.

\begin{proposition}[cf.\ {\cite{ours2022semismooth}}] \label{prop:duality}
  Let $V:=\Hzdiv(\Omega)^m$, $W:=L^2(\Omega)^{m}$. Problem \eqref{eq:model-motivation} is dual to
  \begin{equation} \label{eq:model-motivation-predual}
    \inf_{p \in K} \big\{
      \D(p) := \tfrac{1}{2} \|\LambdaStar p - T^* g\|_{B^{-1}}^2
    \big\},
  \end{equation}
  where
  $K :=  \{p \in V: |p(x)|_{F} \le \lambda \text{ f.a.e. } x\in\Omega\}$,
  $\LambdaStar: V \to W$, $\LambdaStar p = \Div p$,
  $T^*:L^2(\Omega) \to W$ is the adjoint operators of $T$,
  $B: W \to W$
  denotes the operator $B := \alpha_2 T^* T + \beta I$ and
  the norm is given by $\|u\|_{B^{-1}}^2 := \<u,
  {B^{-1}}u\>_{W}$ for $u\in W$, where $\<\cdot,\cdot\>_{W}$ denotes the $W$-inner product.

  The unique solution $\hat u$ of \eqref{eq:model-motivation}
  is related to any solution $\hat p$ of \eqref{eq:model-motivation-predual} by
  \begin{equation} \label{eq:duality-opt}
    \hat u = B^{-1}(-\LambdaStar \hat p + T^* g)
    \qquad\text{and}\qquad
    \forall p \in K: \<\LambdaOp \hat u, p - \hat p\>_{V^*, V} \le 0,
  \end{equation}
  where $\LambdaOp:W^* \to V^*$ is the adjoint operators of $\LambdaStar$ and $\<\cdot, \cdot\>_{V^*, V}$ denotes the duality pairing between $V$ and its dual space $V^*$.
\end{proposition}

Some comments on \cref{prop:duality} are in order: 
To guarantee the existence of a minimizer of \eqref{eq:model-motivation-predual} we assume that the bilinear form $a_B: W \times W \to \R$, $a_B(u,v) := \alpha_2 \<Tu, Tv\>_{W} + \beta \<u, v\>_{W}$  is coercive \cite[Theorem 3.5]{ours2022semismooth}, i.e.\  $a_B(u, u) \ge c_B \|u\|_{W}^2$ with coercivity constant $c_B > 0$ and $\|\cdot\|_W$ being the norm induced by the $W$-inner product. In particular this implies the coercivity of $\D$, i.e., for any feasible sequence $(p^n)_{n\in\N}
\subset K$
\begin{equation} \label{eq:coerciveness}
  \|p^n\|_V \to \infty
  \implies
  \D(p^n) \to \infty
\end{equation} 
with $\|\cdot\|_V$ being the norm associated to $V$ \cite{ours2022semismooth}. Further the coercivity of $a_B$ guarantees the invertibility of $B$ \cite{ours2022semismooth}. 
\Cref{prop:duality} allows one to solve for $\hat p$ in the predual domain of
\eqref{eq:model-motivation-predual} and to later assemble the original solution
$\hat u$ using the optimality relation \eqref{eq:duality-opt}.

Here and in the sequel we write $V$ and $W$ instead of $\Hzdiv(\Omega)^m$ and $L^2(\Omega)^{m}$ as the below presented algorithms (see \cref{alg:ddpar} and \cref{alg:ddseq}) as well as the associated theory also holds for problems of type  \eqref{eq:model-motivation-predual} in a general Hilbert space setting, see \cref{rem:general:setting} below.

Based on a dualization as in \cref{prop:duality} for the setting $T=I$, $\beta=0$ and $m=1$ in \cite{ChaTaiWanYan, HinLan2015_1} convergent overlapping and nonoverlapping domain decomposition methods are introduced. While the convergence in \cite{HinLan2015_1} for a nonoverlapping splitting is proven in a discrete setting, in \cite{ChaTaiWanYan} for an overlapping splitting even a convergence rate in a continuous setting is guaranteed.  These two papers allowed to derive overlapping \cite{LangerGaspoz:19} and nonoverlapping \cite{LeeNam} domain decomposition methods for the primal problem \eqref{eq:model-motivation} in the respective setting and for $T=I$, $\beta=0$ and $m=1$, together with a convergence analysis which ensures that a minimizer of the global problem is indeed approached. 
Since then a series of splitting techniques for total variation minimization have been presented in the literature \cite{LeeNamPark2019, LeeParkPark2019, Lee2019fast, LeePark2019, LiZhangChangDuan2021, park2020additive, park2020overlapping, park2021accelerated}.
For an introduction to domain decomposition approaches for total variation minimization we refer the reader to \cite{Lan2021,LeePark:20}.

In this paper we generalize the overlapping splitting method from \cite{ChaTaiWanYan}, which is restricted to denoising, to the more general problem \eqref{eq:model-motivation-predual} where $T$ can be any arbitrary linear and bounded operator, and hence to applications like inpainting and calculating the optical flow in image sequences. Further while the analysis of the decomposition method in \cite{ChaTaiWanYan} assumes exact local minimization, in our case an approximate local minimization is sufficient. This requires a new convergence analysis which differs significantly from the one in \cite{ChaTaiWanYan}. Moreover we are even able to improve the convergence rate of \cite{ChaTaiWanYan} by a constant. We provide a particular example in which the subproblems are approximated (solved inexactly) by so-called \emph{surrogate} functionals, as in \cite{ForLanSch2010, ForSch, HinLan2014}. For solving the subspace problems we adjust the semi-implicit algorithm by Chambolle \cite{Chambolle:04} to our setting. While the algorithm in \cite{Chambolle:04} is derived in a discrete setting and for image denoising problems only (i.e., $T=I$), we adjust it to our problem, where $T$ might be any bounded and linear operator as already mentioned above, and Hilbert space setting.

We would like to mention that recently in \cite{park2020additive} a general framework for analysing additive Schwarz methods of convex optimization problems as gradient methods has been presented.
For the special case of parallel decomposition their analysis covers
ours, while we extend our results to sequential decomposition which
\cite{park2020additive} does not cover.
We also consider a slightly different notion of approximate minimization (see
\cref{def:approxmin}) for the local subproblems which does not seem to map to
the approximate notion considered in \cite{park2020additive} in an obvious way.

The rest of the paper is organized as follows: In \cref{sec:fundaments} we introduce the setting of our decomposition which is based on the definition of a partition of unity operator. The proposed parallel and sequential decomposition algorithms are described in \cref{sec:algorithm} and their convergence analysis is presented in \cref{sec:convergence-analysis}. In particular we prove the convergence of the presented algorithms to a solution of the global problem together with a convergence rate. In \cref{sec:comaprison} we compare our findings with related results already presented in the literature. As our proposed decomposition algorithms allow for approximate solutions of the subproblems, in \cref{sec:surrogate-technique} we present a concrete example for such a case utilizing the surrogate technique. For solving the constituted subdomain problems we describe in \cref{sec:semi-implicit-algorithm} the semi-implicit algorithm of Chambolle \cite{Chambolle:04} in a general Hilbert space setting for our type of problems. In \cref{sec:numerics} numerical experiments are presented verifying the theoretical sublinear convergence of the proposed decomposition algorithms as well as showing the practical behaviour. We conclude in \cref{sec:conclusion} with some final remarks.

\section{Fundamentals}\label{sec:fundaments}

As we are presenting a decomposition method for \eqref{eq:model-motivation-predual} in the sequel we use the notations and definitions of \cref{prop:duality}. 
Further for a bounded linear operator $A:H_1\to H_2$ between two Hilbert spaces $H_1$ and $H_1$ we use $\|A\|$ for the respective operator norm.

\subsection{Decomposition Setting}

We will analyse a decomposition algorithm for problem \eqref{eq:model-motivation-predual} that
requires a suitable partition of unity respecting the closed convex set $K$ or
more precisely: bounded linear operators $\theta_i: V \to V$, $i = 1, \dotsc,
M$, $M \in \N$ such that
\begin{align} \label{eq:partition-of-unity}
  I = \sum_{i=1}^M \theta_i  \qquad\text{and}\qquad
  K = \sum_{i=1}^M \theta_i K.
\end{align}
Note here that, since $\theta_i$ is a bounded linear operator,
$\theta_i K = \{\theta_i p: p \in K\} \subset V$ stays
closed and convex.

The requirements for the partition given in \eqref{eq:partition-of-unity} are
in particular fulfilled by the following domain decomposition formulation.
Let $\Omega_i$, $i = 1, \dotsc, M$, $M \in \N$ be bounded open sets with Lipschitz boundary such that $\bigcup_{i=1}^M \Omega_i = \Omega$.
Denote by $\tilde \theta_i: \Omega \to [0,1]$, $i = 1, \dotsc, M$ a partition of unity satisfying
\begin{enumerate}[(i)]
  \item
    $\tilde \theta_i \in W^{1,\infty}(\Omega)$,
  \item
    $1 = \sum_{i=1}^M \tilde \theta_i$,
  \item
    $\supp \tilde \theta_i \subset \overline{\Omega}_i$.
\end{enumerate}
We then define the partition of unity operator $\theta_i: V \to V$ by pointwise multiplication
\begin{equation} \label{eq:partition-of-unity-tv}
  (\theta_i p)(x) := \tilde \theta_i(x) p(x),
\end{equation}
for all $p \in V$.

\begin{lemma} \label{lem:partition-of-unity-requirements}
  The partition of unity operators $(\theta_i)_{i=1}^M$ defined in \eqref{eq:partition-of-unity-tv} satisfy the requirements of \eqref{eq:partition-of-unity},
  where $K$ is given by \cref{prop:duality}.
  \begin{proof}
    Linearity of $\theta_i$, $i=1,\dotsc,M$, is inherited from the pointwise multiplicative definition in \eqref{eq:partition-of-unity-tv}.
    Let $p \in V = \Hzdiv(\Omega)^m$,
    then
    \begin{equation*}
      \begin{aligned}
      \|\tilde \theta_i p\|_{L^2}
      &\le \|\tilde \theta_i\|_{L^\infty} \|p\|_{L^2}, \\
      \|\Div (\tilde \theta_i p)\|_{L^2}
      &=\|\nabla \tilde \theta_i p + \tilde \theta_i \Div p\|_{L^2} \\
      &\le \|\nabla \tilde \theta_i\|_{L^\infty}\|p\|_{L^2} + \|\tilde \theta_i\|_{L^\infty} \|\Div p\|_{L^2},
      \end{aligned}
    \end{equation*}
    and thus, since $\tilde \theta_i \in W^{1,\infty}(\Omega)$ and in
    particular $\nabla \tilde \theta_i \in L^\infty(\Omega)$, we have proven that $\theta_i:V \to V$ is indeed well-defined and bounded.

    Due to the pointwise nature of \eqref{eq:partition-of-unity-tv} we see that for $p \in V$:
    \begin{align*}
      \Big(\sum_{i=1}^M \theta_i p\Big)(x)
      = \sum_{i=1}^M \tilde \theta_i(x) p(x)
      = p(x)
    \end{align*}
    which shows $I = \sum_{i=1}^M \theta_i$.

    We have $K \subset \sum_{i=1}^M \theta_i K$ per definition.
    To show the other inclusion let $p^i \in \theta_i K$, $i = 1, \dotsc, M$.
    Then we see that in a pointwise fashion
    \begin{align*}
      \Big|\sum_{i=1}^M p^i(x)\Big|_{F}
      \le \sum_{i=1}^M |p^i(x)|_F
      \le \sum_{i=1}^M \tilde \theta_i(x) \lambda
      \le \lambda,
    \end{align*}
    thus showing that $p^i \in K$.
  \end{proof}
\end{lemma}

\section{Algorithm}\label{sec:algorithm}

Let us first introduce our notion of approximate minimization.

\begin{definition} \label{def:approxmin}
  For $q \in V$, $\rho \in (0,1]$ we call
  \begin{align} \label{eq:approxmin}
    \argmin_{p\in K}^{\rho,q} \D(p)
    := \Big\{\tilde p \in K: \D(q) - \D(\tilde p) \ge \rho(\D(q) - \D(\hat p)), \hat p \in \argmin_{p\in K} \D(p) \Big\}
  \end{align}
  the set of $\rho$-approximate minimizers of $\D$ on $K$ with respect to $q$.
\end{definition}

The condition in \eqref{eq:approxmin} means that the improvement in functional value needs to be at least within a constant factor of the remaining difference in functional value towards a true minimizer.
For $\rho = 1$ and arbitrary $q \in V$ this reduces to the usual notion of minimizers.

We present the decomposition procedures in \cref{alg:ddpar,alg:ddseq}.
\begin{algorithm}
  \caption{Parallel decomposition} \label{alg:ddpar}
  \begin{algorithmic}[1]
    \Require{$p^0\in K$ and $\sigma \in (0,\frac{1}{M}]$, $\rho \in (0, 1]$}
    \For{$n=0,1,2,\ldots$}
      \For{$i=1,\ldots,M$}
        \State{$\tilde{v}_i^n \in \argmin_{v_i \in \theta_i K}^{\rho,\theta_i p^n} \D\big(p^n + (v_i-\theta_i p^n)\big)$}
      \EndFor
      \State{$p^{n+1} = p^n + \sum_{i=1}^M \sigma (\tilde{v}_i^n - \theta_i p^n)$}
    \EndFor
  \end{algorithmic}
\end{algorithm}

\begin{algorithm}
  \caption{Sequential decomposition} \label{alg:ddseq}
  \begin{algorithmic}[1]
    \Require{$p^0\in K$ and $\sigma \in (0,1]$, $\rho \in (0, 1]$}
    \For{$n=0,1,2,\ldots$}
      \State{$p_0^n = p^n$}
      \For{$i=1,\ldots,M$}
        \State{$\tilde{v}_i^n \in \argmin_{v_i \in \theta_i K}^{\rho,\theta_i p^n} \D\big(p_{i-1}^n + (v_i-\theta_i p^n)\big)$}
        \State{$p_i^n = p_{i-1}^n + \sigma(\tilde v_i^n - \theta_i p^n)$}
      \EndFor
      \State{$p^{n+1} = p_{M}^n$}
    \EndFor
  \end{algorithmic}
\end{algorithm}

To treat both algorithms in a similar way, we use the convention
$p_{i-1}^n := p^n$ for
\cref{alg:ddpar} independent of $i \in \{1, \dotsc, M\}$.
Having defined $\tilde v_i^n \in \theta_i K$ for $i \in \{1, \dotsc, M\}$ we
also set $\tilde p_i^n := p_{i-1}^n + (\tilde v_i^n - \theta_i p^n) \in K$.

We observe that in each step $n \in \N_0$, $i \in \{1, \dotsc, M\}$ of
\cref{alg:ddpar,alg:ddseq} the subproblem
\begin{equation} \label{eq:algpdd-subproblem}
  \inf_{v_i \in \theta_i K} \tfrac{1}{2} \|\LambdaStar v_i - f_i^n\|_{B^{-1}}^2
\end{equation}
with $f_i^n = T^*g - \LambdaStar (p_{i-1}^n - \theta_i p^n)$ needs to be solved approximately.

For a locally acting operator $B^{-1}$ and suitable $\theta_i$ these problems may be solved on $\supp(\theta_i) \subset \overline{\Omega}$.
If $B^{-1}$ on the other hand is global then in order to avoid having to solve
the subproblems globally on $\Omega$ a surrogate technique will be introduced
in \cref{sec:surrogate} below.

\section{Convergence Analysis}\label{sec:convergence-analysis}

In this section we analyse \cref{alg:ddpar,alg:ddseq} with respect to their convergence. In particular we first show monotonicity of the energy with respect to the iterates followed by the main results. Subsequently we collect useful statements which finally enable us to prove our main result at the end of this section.
 
\subsection{Monotonicity}
We first establish monotonicity of the iterates.

\begin{lemma} \label{lem:stepdecrease}
  The iterates $(p^n)_{n\in\N}$ of \cref{alg:ddpar,alg:ddseq} with corresponding constraints on $\sigma$ satisfy
  \begin{align*}
    \D(p^n) - \D(p^{n+1})
    &\ge \rho\sigma \sum_{i=1}^M \big( \D(p_i^n) - \D(\hat p_i^n) \big)
    \ge 0
  \end{align*}
  where $\hat p_i^n = p_{i-1}^n + (\hat v_i^n - \theta_i p^n)$, $\hat v_i^n \in
  \argmin_{v_i \in \theta_i K} \D(p_{i-1}^n + (v_i - \theta_i p^n))$
  denotes any exact minimizer in the $i$-th substep of the corresponding
  algorithm.
  The non-negative sequence $(\D(p^n))_{n\in\N}$ is in particular monotonically decreasing
  and thus convergent.

  \begin{proof}
    The update step for $p^{n+1}$ in the parallel case of \cref{alg:ddpar} is given as
    \begin{align*}
      p^{n+1}
      = p^n + \sigma \sum_{i=1}^M (\tilde v_i^n - \theta_i p^n) 
      = (1- \sigma M) p^n  + \sigma \sum_{i=1}^M \big( p^n + (\tilde v_i^n - \theta_i p^n) \big).
    \end{align*}
    We denote $\tilde p_i^n = p_{i-1}^n + (\tilde v_i^n - \theta_ip^n) = p^n + (\tilde v_i^n - \theta_i p^n)$.
    Since $\sigma \in (0, \frac{1}{M}]$, convexity of $\D$ yields
    \begin{align*}
      \D(p^{n+1}) \le (1-\sigma M) \D(p^n) + \sigma \sum_{i=1}^M \D\big( \tilde p_i^n \big).
    \end{align*}
    We use this and the definition of $\tilde v_i^n$ to estimate
    \begin{align*}
      \D(p^n) - \D(p^{n+1})
      &\ge \sigma M \D(p^n) - \sigma \sum_{i=1}^M \D\big( \tilde p_i^n \big)
      = \sigma \sum_{i=1}^M \Big( \D(p^n) - \D\big( \tilde p_i^n \big) \Big) \\
      &\ge \rho \sigma \sum_{i=1}^M \Big( \D(p^n) - \D\big( \hat p_i^n \big) \Big),
    \end{align*}
    where we denoted $\hat p_i^n = p_{i-1}^n + (\hat v_i^n - \theta_i p^n) = p^n + (\hat v_i^n - \theta_i p^n)$
    in the last inequality.

    For the sequential case of \cref{alg:ddseq} we have similarly
    \begin{align*}
      p_i^n
      &= p_{i-1}^n + \sigma(\tilde v_i^n - \theta_i p^n) 
      = (1 - \sigma) p_{i-1}^n + \sigma (p_{i-1}^n - (\tilde v_i^n - \theta_i p^n)) \\
      &= (1 - \sigma) p_{i-1}^n + \sigma \tilde p_i^n
    \end{align*}
    and thus $\D(p_i^n) \le (1-\sigma) \D(p_{i-1}^n) + \sigma \D(\tilde p_i^n)$.
    Rewriting we see that
    \begin{equation} \label{eq:dd-energy-difference-stepsize}
      \begin{aligned}
        \D(p_{i-1}^n) - \D(p_i^n)
        \ge \sigma (\D(p_{i-1}^n) - \D(\tilde p_i^n)) 
        \ge \rho\sigma (\D(p_{i-1}^n) - \D(\hat p_i^n)),
      \end{aligned}
    \end{equation}
    where we again used the definition of $\tilde v_i^n$ in the second inequality.
    A telescope sum over $i = 1, \dotsc, M$ then yields
    \begin{align*}
      \D(p^n) - \D(p^{n+1})
      =\sum_{i=1}^M (\D(p_{i-1}^n) - \D(p_{i}^n)) 
      \ge \rho\sigma \sum_{i=1}^M(\D(p_{i-1}^n) - \D(\hat p_i^n)).
\end{align*}
  \end{proof}
\end{lemma}

In particular, \cref{lem:stepdecrease} shows monotonicity of energies, i.e.\
$\D(p^n) \ge \D(p^{n+1})$.
Because of the coercivity assumption \eqref{eq:coerciveness},
the set of iterates $\{p^n: n \in \N_0\} \subset V$ is therefore bounded.
We thus denote for some fixed minimizer $\hat p \in K$ of $\D$ the finite radius
\begin{align} \label{eq:dd_iterates_radius}
  R_{\hat p} := \sup \{ \|p - \hat p\|_V : p \in K, \D(p) \le \D(p^0) \} < \infty.
\end{align}

\subsection{Main Results}
Now we are able to state our main convergence result providing a convergence rate of the proposed algorithms to a minimizer of the original (global) problem.

\begin{theorem} \label{thm:ddrate}
  Let $(p^n)_{n\in\N_0}$ be the iterates from either one of \cref{alg:ddseq,alg:ddpar} and
  let $\hat p \in K$ denote a minimizer of $\D$.
  \Cref{alg:ddpar,alg:ddseq} converge in the sense that $\D(p^n) \to \D(\hat p)$.
  More specifically,
  \begin{align*}
    \D(p^n) - \D(\hat p)
    &\le \begin{cases}
      (1 - \tfrac{\rho\sigma}{2\alpha})^n \big(\D(p^{0}) - \D(\hat p)\big) & \text{if $n \le n_0$} \\
      \tfrac{2\Phi^2}{\rho \sigma}\alpha^2 (n - n_0 +1)^{-1} & \text{if $n \ge n_0$,}
    \end{cases}
  \end{align*}
  where $\alpha := 1 + M\sigma\sqrt{2-\rho+2\sqrt{1-\rho}}$ for \cref{alg:ddseq}
  and $\alpha := 1$ for \cref{alg:ddpar},
  $\Phi := \sqrt{\|B^{-1}\|} \NormLambda C_\theta R_{\hat p}$ with $C_\theta := \left(\sum_{i=1}^M \|\theta_i\|^2\right)^{\frac{1}{2}}$ and
  $n_0 := \min \{n \in \N_0 : \D(p^{n}) - \D(\hat p) < \Phi^2 \alpha\}$.
  \begin{proof}
  The proof is deferred to \cref{sec:proof}.
  \end{proof}
\end{theorem}

The difference in the predual energy can be related to the $L^2$-error of the primal variable in the following way:

\begin{proposition} \label{prop:dual-energy-strong-convexity}
  Let $\hat p \in V$ be a minimizer of \eqref{eq:model-motivation-predual}
  and $\hat u \in W$ be the minimizer of \eqref{eq:model-motivation}.
  Then
  for all $p \in V$ and $u := B^{-1}(-\LambdaStar p + T^* g)$ we have
  \begin{align*}
    \tfrac{c_B}{2} \|u - \hat u\|_W^2
    \le \D(p) - \D(\hat p).
  \end{align*}

  \begin{proof}
    Due to coercivity of $a_B$ we have for $v \in W$
    \begin{align*}
      c_B \|B^{-1}v\|_W^2
      \le a_B(B^{-1}v,B^{-1}v)
      &= \big\<(T^*T + \beta I) B^{-1}v, B^{-1}v\big\>_W \\
      &= \<v, B^{-1}v\>_W
      = \|v\|_{B^{-1}}^2.
    \end{align*}
    By expanding the quadratic functional $\D$ at $\hat p$
    and using optimality of $\hat p$, i.e.\ $\<\D'(\hat p), p - \hat p\> \ge 0$,
    we then see that
    \begin{align*}
      \D(p) - \D(\hat p)
      &= \<\D'(\hat p), p - \hat p\>_V
        + \tfrac{1}{2} \<\LambdaOp B^{-1} \LambdaStar (p - \hat p), p - \hat p\>_V \\
      &\ge \tfrac{1}{2}\|\LambdaStar (p - \hat p)\|_{B^{-1}}^2
      \ge \tfrac{c_B}{2} \|B^{-1} \LambdaStar (p - \hat p)\|_W^2
      = \tfrac{c_B}{2} \|u - \hat u\|_W^2,
    \end{align*}
    since due to \cref{prop:duality} $\hat u$ is given by $\hat u =
    B^{-1}(-\LambdaStar \hat p + T^* g)$.
  \end{proof}
\end{proposition}

\subsection{Collection of Useful Results}
Here we collect some statements which we use to proof \cref{thm:ddrate}.
\begin{definition}
  For $p, q \in V$ we introduce the notation
  \begin{align*}
    \<p, q\>_* &:= \<\LambdaOp B^{-1} \LambdaStar p, q\>_{V},&
    \|p\|_* &:= \sqrt{\<p,p\>_*}.
  \end{align*}
\end{definition}

Note that we have $\|p\|_*^2 = \|\LambdaStar p\|_{B^{-1}}^2$ in particular and
that $\<\cdot, \cdot\>_*$ and $\|\cdot\|_*$ are not necessarily positive definite.

\begin{lemma} \label{Lemma22}
  Let $\D': V \to V$ be the Fréchet derivative of $\D$.
  For any $p,q,r\in V$ we have
  \begin{enumerate}[(i)]
    \item
      $\D(p)-\D(q) = \<\D'(q), p - q\>_{V} + \tfrac{1}{2}\|p-q\|_*^2$,
    \item
      $\< \D'(p) - \D'(q), r \>_{V} = \<p - q, r\>_*$,
  \end{enumerate}
\end{lemma}
\begin{proof}
  \begin{enumerate}[(i)]
    \item
      We expand the quadratic functional $\D$ at $q$ to obtain
      \begin{align*}
        \D(p)
        &= \D(q) + \<\D'(q), p - q\>_{V} + \tfrac{1}{2}\<\LambdaOp B^{-1} \LambdaStar
        (p - q), p - q\>_{V} \\
        &= \D(q) + \<\D'(q), p - q\>_{V} + \tfrac{1}{2}\|p - q\|_*^2
      \end{align*}
    \item
     We see directly
      \begin{align*}
        \<\D'(p) - \D'(q), r\>_{V}
	&= \<\LambdaOp B^{-1}(\LambdaStar p - T^* g) - \LambdaOp B^{-1}(\LambdaStar q -
        T^* g), r\>_{V} \\
        &= \<\LambdaOp B^{-1} \LambdaStar (p-q), r\>_{V} 
	= \<p-q,r\>_*.
      \end{align*}~
  \end{enumerate}
\end{proof}

We note that \cref{Lemma22} actually holds true for any quadratic
functional.

\begin{lemma} \label{lem:seqrate}
  Let $c > 0$ and $(a_k)_{k \in \N_0} \subset \R^+$ be a sequence such that for all $k \in \N_0$:
  \begin{align*}
    a_k - a_{k+1} \ge c a_k^2.
  \end{align*}
  Then $\lim_{k\to\infty} a_k \to 0$ with rate
  \begin{align*}
    0 < a_k < \frac{1}{ck + \frac{1}{a_0}} < \frac{1}{ck}
  \end{align*}
  for all $k \in \N$.
  \begin{proof}
    We proceed similar to \cite{both2022rate}.
    Since the iterates $a_k$, $k\in \N_0$ are monotonically decreasing, we can write for $k \in \N_0$:
    \begin{align*}
      \frac{1}{a_{k+1}} - \frac{1}{a_k}
      = \frac{a_k - a_{k+1}}{a_{k+1}a_k}
      \ge \frac{c a_k}{a_{k+1}}
      > c
    \end{align*}
    and use it to reduce the telescope sum for $k > 0$:
    \begin{align*}
      \frac{1}{a_k}
      = \sum_{j=0}^{k-1} \Big( \frac{1}{a_{j+1}} - \frac{1}{a_j} \Big) + \frac{1}{a_0}
      > ck + \frac{1}{a_0}.
    \end{align*}
    Inverting the inequality yields the statement.
  \end{proof}
\end{lemma}

\begin{lemma} \label{lem:estsplit}
  Let $a, b > 0$, $c, x, y \ge 0$ such that for all $\mu \in (0, 1]$ the inequality
  \begin{align*}
    y \le a \mu + \frac{b}{\mu} x + c \sqrt{x}
  \end{align*}
  holds. Then the following split inequality holds:
  \begin{align} \label{eq:split_inequality}
    y &\le \begin{cases}
      (2b + c \frac{\sqrt{b}}{\sqrt{a}})x & \text{if $x > \frac{a}{b}$}, \\
      (2\sqrt{ab} + c) \sqrt{x} & \text{if $x \le \frac{a}{b}$},
    \end{cases}
    \intertext{or equivalently}
    x &\ge \begin{cases}
      (2b + c \frac{\sqrt{b}}{\sqrt{a}})^{-1} y & \text{if $y > 2a +
      c\frac{\sqrt{a}}{\sqrt{b}}$}, \\
        (2\sqrt{ab} + c)^{-2} y^2 & \text{if $y \le 2a +
        c\frac{\sqrt{a}}{\sqrt{b}}$}.
    \end{cases} \notag
  \end{align}
  \begin{proof}
    If $x > \frac{a}{b}$ we choose $\mu = 1$ to arrive at
    \begin{align*}
      y
      \le a + bx + c\sqrt{x}
      < 2bx + c \sqrt{x}
      \le (2b + c \tfrac{\sqrt{b}}{\sqrt{a}}) x.
    \end{align*}
    Otherwise we minimize the expression by choosing
    $\mu = \tfrac{\sqrt{b}}{\sqrt{a}} \sqrt{x}$ and get
    \begin{align*}
      y
      \le a \tfrac{\sqrt{b}}{\sqrt{a}}\sqrt{x} + b \tfrac{\sqrt{a}}{\sqrt{b}} \sqrt{x} + c \sqrt{x}
      = (2\sqrt{ab} + c)\sqrt{x}.
    \end{align*}
    Both statements together yield the estimate.

    Noting that the right-hand side of estimate \eqref{eq:split_inequality}
    is continuous and monotone
    in $x$, the case distinction can equivalently be written in terms of $y$
    by splitting at $x = \tfrac{a}{b}$, $y = (2b + c \tfrac{\sqrt{b}}{\sqrt{a}})\tfrac{a}{b} = 2a + c\sqrt{a}{b}$.
    Seperately solving the inequalities for $x$ thus yields the equivalent
    representation.
  \end{proof}
\end{lemma}

\begin{lemma} \label{lem:normbounds}
  We have
  \begin{align*}
    \sum_{i=1}^M \|\theta_i p\|_*^2
    &\le \|B^{-1}\|\NormLambda^2 C_\theta^2 \|p\|_V^2.
  \end{align*}
  \begin{proof}
    Application of the Cauchy-Schwarz inequality yields
    \begin{align*}
      \sum_{i=1}^M \|\theta_i p\|_*^2
      &= \sum_{i=1}^M \<\LambdaStar \theta_i p, B^{-1} \LambdaStar \theta_i p\>_W \\
      &\le \sum_{i=1}^M \|B^{-1}\| \NormLambda^2 \|\theta_i\|^2 \|p\|_V^2
      = \|B^{-1}\| \NormLambda^2
        \Big(\sum_{i=1}^M \|\theta_i\|^2\Big) \|p\|_V^2.
    \end{align*}
  \end{proof}
\end{lemma}

In the following we employ ideas from alternating minimization
\cite{both2022rate} to achieve a convergence rate estimate. 
\begin{lemma} \label{lem:dd-step-distance-energy}
  We may estimate the step distance in terms of the corresponding energy change
  as follows:
  \begin{align*}
    \tfrac{1}{2} \|p_{i-1}^n - p_i^n\|_*^2
    \le \tfrac{\sigma}{\rho}\big(2 - \rho + 2 \sqrt{1-\rho}\big)
      (\D(p_{i-1}^n) - \D(p_i^n)).
  \end{align*}
  \begin{proof}
    Let $\omega > 0$ to be chosen later and denote
    $\tilde p_i^n := p_{i-1}^n + (\tilde v_i^n - \theta_i p^n)$.
    \begin{align*}
      \tfrac{1}{2\sigma^2}\|p^n_{i-1} - p^n_i\|_*^2
      &=\tfrac{1}{2\sigma^2}\|\sigma(\tilde v_i^n - \theta_i p^n)\|_*^2 \\
      &=\tfrac{1}{2}\|p^n_{i-1} - \tilde p^n_i\|_*^2 \\
      &\le \tfrac{1}{2}\Big(
        (1+\omega)\|p^n_{i-1} - \hat p^n_i\|_*^2 +
        (1+\omega^{-1})\|\tilde p^n_i - \hat p^n_i\|_*^2 \Big) \\
      &\le (1+\omega)(\D(p^n_{i-1}) - \D(\hat p^n_i)) +
        (1+\omega^{-1})(\D(\tilde p^n_i) - \D(\hat p^n_i))
        \tag{\cref{Lemma22} (i) and optimality}\\
      &\le \tfrac{1+\omega}{\rho}(\D(p^n_{i-1}) - \D(\tilde p^n_i)) +
        \tfrac{(1+\omega^{-1})(1-\rho)}{\rho} (\D(p^n_{i-1}) - \D(\tilde p^n_i))
        \tag{due to \eqref{eq:approxmin}}\\
      &= \tfrac{1}{\rho} \big(1+\omega + (1+\omega^{-1})(1-\rho)\big) ( \D(p^n_{i-1}) - \D(\tilde p^n_i)) \\
      &\le \tfrac{1}{\sigma\rho} \big(1+\omega + (1+\omega^{-1})(1-\rho)\big) (\D(p^n_{i-1}) - \D(p^n_i))
        \tag{using \eqref{eq:dd-energy-difference-stepsize}}.
    \end{align*}
    Choosing $\omega := \sqrt{1-\rho}$ so as to minimize the expression we
    arrive at
    \begin{align*}
      \tfrac{1}{2\sigma^2}\|p^n_{i-1} - p^n_i\|_*^2
      \le \tfrac{1}{\sigma\rho}\big(2 - \rho + 2\sqrt{1-\rho}\big)
        (\D(p^n_{i-1}) - \D(p^n_i)).
    \end{align*}
  \end{proof}
\end{lemma}

\begin{proposition} \label{prop:ratestep}
  Let $(p^n)_{n\in\N_0}$ be the iterates from either one of \cref{alg:ddseq,alg:ddpar} and
  let $\hat p \in K$ denote a minimizer of $\D$.

  Then $\D(p^n) \to \D(\hat p)$ as $n \to \infty$ owing to
  \begin{align*}
    \D(p^n) - \D(\hat p) \le
    \begin{cases}
      \tfrac{2}{\rho\sigma}\alpha \big(\D(p^n) - \D(p^{n+1})\big) & \text{if $\D(p^n) - \D(p^{n+1}) > \tfrac{1}{2}\sigma\rho\Phi^2$}, \\
      \sqrt{\tfrac{2}{\rho\sigma}}\Phi\alpha\sqrt{\D(p^n) - \D(p^{n+1})} & \text{else},
    \end{cases}
  \end{align*}
  where $\alpha := 1 + M\sigma\sqrt{2-\rho+2\sqrt{1-\rho}}$ for \cref{alg:ddseq}
  and $\alpha := 1$ for \cref{alg:ddpar},
  and $\Phi := \sqrt{\|B^{-1}\|} \NormLambda C_\theta R_{\hat p}$.
  \begin{proof}
    Using convexity we expand
    \begin{align} \label{eq:dd_step_expansion}
      \D(p^n) - \D(\hat p)
      &\le \<\D'(p^n), p^n - \hat p\>_V \notag \\
      &= \sum_{i=1}^M \<\D'(p^n), \theta_i(p^n - \hat p)\>_V \notag \\
      &= \sum_{i=1}^M \Big(\<\D'(p^n_{i-1}), \theta_i(p^n - \hat p)\>_V
        + \sum_{j = 1}^{i - 1} \<\D'(p^n_{j-1}) - \D'(p^n_j), \theta_i(p^n - \hat p)\>_V \Big).
    \end{align}
    Let $\Phi_n := (\sum_{i=1}^M \|\theta_i(p^n - \hat p)\|_*^2)^{\frac{1}{2}}$, $\hat v^n_i \in \argmin_{v_i \in \theta_i K} \D(p_{i-1}^n + (v_i - \theta_i p^n))$ and $\hat p_i^n := p_{i-1}^n + (\hat v_i^n - \theta_i p^n)$.
    We now estimate the first summand in the expansion above:
    \begin{align*}
      &\sum_{i=1}^M \<\D'(p^n_{i-1}), \theta_i(p^n - \hat p)\>_V \\
      &\qquad= \tfrac{1}{\mu} \sum_{i=1}^M \<\D'(p^n_{i-1}), \mu\theta_i(p^n - \hat p)\>_V \\
      &\qquad= \tfrac{\mu}{2} \sum_{i=1}^M \|\theta_i(p^n - \hat p)\|_*^2 +
      \tfrac{1}{\mu} \sum_{i=1}^M \big(\D(p^n_{i-1}) - \D(p^n_{i-1} - \mu\theta_i(p^n -
      \hat p))\big) \tag{\cref{Lemma22} (i)}\\
      &\qquad= \tfrac{\Phi_n^2 \mu}{2} + \tfrac{1}{\mu} \sum_{i=1}^M \Big(\D(p^n_{i-1}) - \D\big(p^n_{i-1} + ((1-\mu)\theta_i p^n + \mu \theta_i \hat p - \theta_i p^n)\big)\Big) \\
      &\qquad\le \tfrac{\Phi_n^2 \mu}{2} + \tfrac{1}{\mu} \sum_{i=1}^M (\D(p^n_{i-1}) - \D(\hat p^n_{i})) \tag{optimality} \\
      &\qquad\le \tfrac{\Phi_n^2 \mu}{2} + \tfrac{1}{\mu\rho\sigma} (\D(p^n) -
      \D(p^{n+1})), \tag{\cref{lem:stepdecrease}}
    \end{align*}
    where optimality was used by realizing that $(1-\mu)\theta_i p^n + \mu \theta_i \hat p \in \theta_i K$.
    For the second summand we see
    \begin{align*}
      &\sum_{i=1}^M \sum_{j=1}^{i-1} \<\D'(p^n_{j-1}) - \D'(p^n_j), \theta_i(p^n - \hat p)\>_V\\
      &\qquad= \sum_{i=1}^M \sum_{j=1}^{i-1} \big\<p^n_{j-1} - p^n_{j}, \theta_i(p^n - \hat
        p)\big\>_* \tag{\cref{Lemma22} (ii)} \\
      &\qquad\le \sum_{i=1}^M \sum_{j=1}^{i-1} \|p^n_{j-1} - p^n_j\|_* \|\theta_i(p^n - \hat p)\|_* \\
      &\qquad\le M \Big(\sum_{j=1}^M \|p^n_{j-1} - p^n_j\|_*^2 \Big)^{\frac{1}{2}}
      \Big(\sum_{i=1}^M \|\theta_i (p^n - \hat p)\|_*^2 \Big)^{\frac{1}{2}} \\
      &\qquad\le M \Phi_n \Big(\sum_{j=1}^M \|p^n_{j-1} - p^n_j\|_*^2 \Big)^{\frac{1}{2}}.
    \end{align*}
    Applying \cref{lem:dd-step-distance-energy} completes the estimate of the
    second summand, yielding
    \begin{align*}
      &\sum_{i=1}^M \sum_{j=1}^{i-1}
        \<\D'(p^n_{i-1}) - \D'(p^n_i), \theta_j(p^n - \hat p)\>_V \\
      &\qquad\le M \Phi_n \sqrt{\tfrac{2\sigma}{\rho} (2-\rho+2\sqrt{1-\rho})}
        \big(\D(p^n) - \D(p^{n+1})\big)^{\frac{1}{2}}.
    \end{align*}
    Combining both estimates and roughly bounding $\Phi_n \le \Phi$ due to \cref{lem:normbounds} we have
    \begin{align*}
      \D(p^n) - \D(\hat p)
      &\le \tfrac{\Phi^2 \mu}{2} + \tfrac{1}{\mu\rho\sigma} \big(\D(p^n) -
      \D(p^{n+1})\big) \\
      &\quad +
      M \Phi \sqrt{\tfrac{2\sigma}{\rho} (2-\rho+2\sqrt{1-\rho})} \big(\D(p^n) - \D(p^{n+1})\big)^{\frac{1}{2}}.
    \end{align*}
    Invoking \cref{lem:estsplit} with $a = \tfrac{\Phi^2}{2}$, $b = \frac{1}{\rho\sigma}$ and $c = M \Phi \sqrt{\tfrac{2\sigma}{\rho} (2-\rho+2\sqrt{1-\rho})}$ yields the split bound with the following coefficients:
    \begin{align*}
      2b + c\sqrt{\tfrac{b}{a}}
      &= \tfrac{2}{\rho\sigma} + M \Phi \sqrt{\tfrac{2\sigma}{\rho} (2-\rho+2\sqrt{1-\rho})} \sqrt{\tfrac{2}{\sigma\rho\Phi^2}} \\
      &= \tfrac{2}{\rho\sigma}(1 + M\sigma\sqrt{2-\rho+2\sqrt{1-\rho}}),
      \\
      2\sqrt{ab} + c
      &= 2 \sqrt{\tfrac{\Phi^2}{2\rho\sigma}} + M \Phi \sqrt{\tfrac{2\sigma}{\rho} (2-\rho+2\sqrt{1-\rho})} \\
      &= \sqrt{\tfrac{2}{\rho\sigma}}\Phi(1 + M\sigma\sqrt{2-\rho+2\sqrt{1-\rho}}),
    \end{align*}
    which concludes the proof for \cref{alg:ddseq}.

    For \cref{alg:ddpar}, examining the proof above, we notice that for the
    parallel version we have $p_i^n = p_{i-1}^n = p^n$ for $i = 1, \dotsc, M-1$
    and thus the second summand in \eqref{eq:dd_step_expansion} vanishes
    completely.
    This allows us to invoke \cref{lem:estsplit} with $c = 0$ and leads to the
    desired statement.
  \end{proof}
\end{proposition}

Now we are able to prove our main result. 

\subsection{Proof of \cref{thm:ddrate}}\label{sec:proof}
We first observe that since $(\D(p^n))_{n \in \N_0}$ is monotonically
    decreasing, $n_0$ is well-defined and we have
    $\D(p^{n}) - \D(\hat p) \ge \Phi^2 \alpha$ for all $n \in \N_0$, $n < n_0$ and likewise $\D(p^{n}) - \D(\hat p) < \Phi^2 \alpha$ for all $n \in \N_0$, $n \ge n_0$.

    We now make use of \cref{prop:ratestep}.
    The equivalence in \cref{lem:estsplit} then yields
    \begin{align*}
      \D(p^{n-1}) - \D(p^{n}) \ge \begin{cases}
	\frac{\rho\sigma}{2\alpha} (\D(p^{n-1}) - \D(\hat p)) & \text{if $n-1 < n_0$} \\
	\frac{\rho\sigma}{2 \Phi^2\alpha^2} (\D(p^{n-1}) - \D(\hat p))^2 & \text{if $n - 1 \ge n_0$}.
      \end{cases}
    \end{align*}
    In the former case we invert the inequality and add $\D(p^{n-1}) - \D(\hat p)$ to arrive at
    \begin{align*}
      \D(p^n) - \D(\hat p) \le (1 - \tfrac{\rho\sigma}{2\alpha}) (\D(p^{n-1}) - \D(\hat p)),
    \end{align*}
    which recursively yields the required statement for all $n \le n_0$.
    In the latter case we may assume without loss of generality that $n_0 = 0$ since we can shift the sequence if necessary.
    Thus for all $n \in \N_0$:
    \begin{align*}
      \D(p^n) - \D(p^{n+1}) \ge \tfrac{\rho\sigma}{2\Phi^2\alpha^2} (\D(p^n) - \D(\hat p))^2.
    \end{align*}
    Invoking \cref{lem:seqrate} with constant $c := \tfrac{\rho\sigma}{2\Phi^2\alpha^2}$ we obtain
    \begin{align*}
      \D(p^n) - \D(\hat p)
      \le \frac{1}{c n + \tfrac{1}{\D(p^0) - \D(\hat p)}}
      \le \frac{1}{c n + \tfrac{1}{\Phi^2\alpha}}
      \le \frac{1}{c n + \tfrac{\rho\sigma}{2\Phi^2\alpha^2}}
      = \frac{1}{c (n+1)}
    \end{align*}
    since $0 \le \sigma, \rho \le 1$ and $\alpha \ge 1$, thereby showing the second inequality, which completes the proof.
\section{Comparison}\label{sec:comaprison}

We conclude that in special cases the results obtained here are either in
agreement with or may improve upon other known estimates. In particular, in the following we compare our findings with the ones in \cite{park2020additive} and \cite{ChaTaiWanYan}. 

\subsection{Gradient Method Framework \cite{park2020additive}}

In the special case of parallel decomposition, i.e. $\alpha = 1$, and exact
local solutions, i.e.\ $\rho = 1$, the framework of \cite{park2020additive} is applicable
to our model and their estimate \cite[Algorithm
4.1]{park2020additive} reproduces ours.
We show this by specializing and transforming their estimate.

Using notation from \cite{park2020additive} we employ
\cite[Algorithm 4.1]{park2020additive}
by setting $E(u) = F(u) + G(u) := \D(u) + \chi_K (u)$, where $\chi_K(u)=0$ if $u\in K$ and $\infty$ otherwise.
The space decomposition is specified by the images of $\theta_k$, $k = 1,
\dotsc, M$, i.e. $V_k := \operatorname{im} \theta_k \subset V$ with $R_k^*: V_k \to V$ then being the inclusion map.
We chose to use exact local solvers, i.e.\ $\rho = 1$ in our notation, since
it is not obvious to us how our notion of approximate minimization map to theirs.
In particular, $d_k$ and $G_k$ are chosen as in \cite[(4.3)]{park2020additive} and $\omega := \omega_0 := 1$.
We now verify \cite[Assumptions 4.1 to 4.3]{park2020additive} in order to
apply \cite[Theorem 4.7]{park2020additive}.
\cite[Assumption 4.1]{park2020additive} is fulfilled due to \cref{Lemma22,lem:normbounds} with $C_{0,K} := C_\theta\NormLambda\sqrt{\|B^{-1}\|}$ and $q := 2$.
We fulfill \cite[Assumption 4.2]{park2020additive} by choosing $\tau_0 := \frac{1}{N}$ (their $\tau$ corresponds to our $\sigma$).
\cite[Assumption 4.3]{park2020additive} is trivialized in the case of exact local solvers.
Applying \cite[Theorem 4.7]{park2020additive} with $C_{q,\tau} = 2$ and $\kappa =
\tfrac{1}{\tau} C_\theta^2 \NormLambda^2\|B^{-1}\|$ yields
\begin{align} \label{eq:dd_comparison_park_estimate1}
  \D(p^1) - \D(\hat p)
  &\le (1 - \sigma(1 - \tfrac{1}{2})) (\D(p^0) - \D(\hat p))
  =(1 - \tfrac{\sigma}{2}) (\D(p^0) - \D(\hat p))
\intertext{
  if $\D(p^0) - \D(\hat p) \ge \tau R_{\hat p}^2 \kappa = \Phi^2$ and
}
    \label{eq:dd_comparison_park_estimate2}
  \D(p^n) - \D(\hat p)
  &\le \tfrac{C_{q,r}R_{\hat p}^2 \kappa}{(n+1)^{q-1}}
  = \tfrac{2\Phi^2}{\sigma} (n+1)^{-1}
\end{align}
otherwise.
Applying estimate \eqref{eq:dd_comparison_park_estimate1} recursively and shifting the sequence by $n_0$
for the estimate \eqref{eq:dd_comparison_park_estimate2} finally yields the formulation
\begin{align*}
  \D(p^n) - \D(\hat p)
  &\le \begin{cases}
    (1 - \tfrac{\sigma}{2})^n \big(\D(p^{0}) - \D(\hat p)\big) & \text{if $n \le n_0$} \\
    \tfrac{2\Phi^2}{\sigma} (n - n_0 +1)^{-1} & \text{if $n \ge n_0$,}
  \end{cases}
\end{align*}
which is in agreement with \cref{thm:ddrate}.

\subsection{Decomposition of the Rudin-Osher-Fatemi Model \cite{ChaTaiWanYan}}

In order to compare with the convergence rate in \cite{ChaTaiWanYan}, we
specialize our model to their setting by chosing $V = \Hzdiv(\Omega)$,
$\LambdaStar = \op{div}: V \to L^2(\Omega)$, $T = I: L^2(\Omega) \to
L^2(\Omega)$, $\beta = 0$ (thus $B = I$) and $\rho = 1$.
Next we introduce some notation from \cite{ChaTaiWanYan}, namely $C_0, \delta >
0$ such that $\|\nabla \tilde \theta_i\|_{L^\infty} \le \tfrac{C_0}{\delta}$
for $i = 1, \dotsc, M$,
c.f.\ \cite[(2.10)]{ChaTaiWanYan}, $\zeta^0 := 2 (\D(p^0) - \D(\hat p))$ (our $\D$
has an additional factor of $\tfrac{1}{2}$) and $N_0
:= \max_{x\in\Omega} | \{i \in \{1, \dotsc, M\} : x \in \Omega_i\} |$.
Then \cite[Theorem 3.1]{ChaTaiWanYan} and \cite[Theorem 3.6]{ChaTaiWanYan}
provide the following estimate:
\begin{align} \label{eq:}
  \tfrac{1}{2} \|u^n - \hat u\|_{L^2}^2
  \le \D(p) - \D(\hat p)
  \le C n^{-1}
\end{align}
where $u^n := -\op{div} p^n + g$, $\hat u := -\op{div} \hat p + g$ and
\begin{align*}
  C := \tfrac{1}{2}\zeta^0 \Big( \tfrac{2}{\sigma} (2M + 1)^2 + 8\sqrt{2} C_0 \lambda
  |\Omega|^{\frac{1}{2}} (\zeta^0)^{-\frac{1}{2}} \frac{M \sqrt{N_0}}{\delta
    \sqrt{\sigma}}  + \sqrt{2} - 1 \Big)^2.
\end{align*}
Note that we used our notation for $M$ and $\sigma$.

In order to compare favorably in this setting, we slightly refine the estimate
$\Phi_n \le \Phi$ from the proof of \cref{prop:ratestep}.
First, we quantify an estimate from the proof of
\cref{lem:partition-of-unity-requirements}.
For all $p \in V$ we have
\begin{align*}
  \sum_{i=1}^M \|\op{div} \theta_i p\|_{L^2}^2
  &\le \sum_{i=1}^M
    \|\nabla \tilde \theta_i \cdot p + \tilde \theta_i \op{div} p\|_{L^2}^2 \\
  &\le \sum_{i=1}^M \Big((1+\omega) \|\nabla \tilde \theta_i \cdot p\|_{L^2}^2
    + (1 + \omega^{-1}) \|\tilde \theta_i \op{div} p\|_{L^2}^2 \\
  &= (1+\omega) \int_\Omega \sum_{i=1}^M |\nabla \tilde \theta_i \cdot p|^2 \di[x] +
    (1+\omega^{-1}) \int_\Omega \sum_{i=1}^M |\tilde \theta_i \op{div} p|^2 \di[x] \\
  &\le (1+\omega) \int_\Omega \Big( \sum_{i=1}^M |\nabla \tilde \theta_i|^2
  \Big) |p|^2 \di[x] \\
  &\qquad +
    (1+\omega^{-1}) \int_\Omega \Big( \sum_{i=1}^M |\tilde \theta_i|^2\Big) |\op{div} p|^2 \di[x] \\
  &\le (1+\omega) N_0 \|\nabla \tilde
    \theta_i\|_{L^\infty}^2 \|p\|_{L^2}^2
    + (1+\omega^{-1}) \|\op{div} p\|_{L^2}^2 \\
  &\le (1+\omega) N_0 \tfrac{C_0^2}{\delta^2} \|p\|_{L^2}^2
    + (1 + \omega^{-1}) \|\op{div} p\|_{L^2}^2,
\end{align*}
for any $\omega > 0$.
The pointwise box-constraints $|p|\le \lambda$ imply $\|p^n - \hat p\|_{L^2}^2 = \int_\Omega |p^n -
\hat p|^2 \di[x] \le (2\lambda)^2 |\Omega|$.
Combining this allows us to estimate
\begin{align*}
  \Phi_n^2
  := \sum_{i=1}^M \|\theta_i (p^n - \hat p)\|_*^2
  &= \sum_{i=1}^M \|\op{div} \theta_i (p^n - \hat p) \|_{L^2}^2 \\
  &\le (1+\omega) N_0 \tfrac{C_0^2}{\delta^2} \|p^n - \hat p\|_{L^2}^2
    + (1 + \omega^{-1}) \|\op{div} (p^n - \hat p)\|_{L^2}^2 \\
  &\le (1+\omega) \cdot 4 \lambda^2 |\Omega| N_0 \tfrac{C_0^2}{\delta^2}
    + (1 + \omega^{-1}) \zeta^0 \\
  &= \Big(2\lambda |\Omega|^{\frac{1}{2}} N_0^{\frac{1}{2}} \tfrac{C_0}{\delta}
   + (\zeta^0)^{\frac{1}{2}}\Big)^2
   =: \tilde \Phi^2
\end{align*}
by optimally choosing $\omega := (4\lambda^2
  |\Omega| N_0 \frac{C_0^2}{\delta^2} )^{-\frac{1}{2}}(\zeta^0)^{\frac{1}{2}}$.
We therefore conclude that in this specific setting \cref{thm:ddrate} holds
true with $\Phi$ replaced by $\tilde \Phi$.
Their and our estimate thus amount to
\begin{align*}
  \tfrac{1}{2} \|u^n - \hat u\|^2
  &\le C n^{-1}, \\
  \tfrac{1}{2} \|u^n - \hat u\|^2
  &\le \tfrac{2\tilde \Phi^2}{\sigma} \alpha^2(n-n_0 + 1)^{-1},
\end{align*}
where for the lower estimate $\alpha$ and $n_0$ are defined as in
\cref{thm:ddrate} and $n \ge n_0$.
Rewriting the involved constants,
\begin{align*}
  C
  &= \bigg(
    \Big(\tfrac{\sqrt{2}}{2} \tfrac{(2M+1)^2}{\sigma} + \sqrt{2} - 1\Big)
        \sqrt{\zeta^0}
    + 8 \tfrac{M}{\sqrt{\sigma}} \lambda |\Omega|^{\frac{1}{2}}
      \sqrt{N_0} \tfrac{C_0}{\delta}
    \bigg)^2, \\
  \tfrac{2\tilde \Phi^2 \alpha^2}{\sigma}
  &\le \tfrac{2(1+\sigma M)^2}{\sigma} \Big(2\lambda |\Omega|^{\frac{1}{2}}
    \sqrt{N_0} \tfrac{C_0}{\delta} + \sqrt{\zeta^0} \Big)^2 \\
  &= \bigg( \sqrt{2} \tfrac{1+\sigma M}{\sqrt{\sigma}} \sqrt{\zeta^0}
    + 2\sqrt{2} \tfrac{1+\sigma M}{\sqrt{\sigma}} \lambda |\Omega|^{\frac{1}{2}}
    \sqrt{N_0} \tfrac{C_0}{\delta} \bigg)^2,
\end{align*}
we see that $\tfrac{2\tilde\Phi^2\alpha^2}{\sigma} \le C$ by comparing the
relevant
terms before $\sqrt{\zeta^0}$ and $\lambda |\Omega|^{\frac{1}{2}} \sqrt{N_0}
\tfrac{C_0}{\delta}$ under the square separately using $0 < \sigma \le 1$ and $M \ge 1$:
\begin{align*}
  \sqrt{2} \tfrac{1+\sigma M}{\sqrt{\sigma}}
  &\le \tfrac{\sqrt{2}}{\sqrt{\sigma}} (1 + M)
  < \tfrac{\sqrt{2}}{\sigma} \tfrac{(2M+1)^2}{2}
  \le \tfrac{\sqrt{2}}{2} \tfrac{(2M+1)^2}{\sigma} + \sqrt{2} - 1 \\
  2\sqrt{2} \tfrac{1+\sigma M}{\sqrt{\sigma}}
  &\le 3 \tfrac{1 + M}{\sqrt{\sigma}}
  < 4 \tfrac{2M}{\sqrt{\sigma}}
  = 8 \tfrac{M}{\sqrt{\sigma}}.
\end{align*}
Consequently, \cref{thm:ddrate} provides a strictly better estimate than
\cite[Theorems 3.1, 3.6]{ChaTaiWanYan} both for sufficiently large $n \in \N$
and for all $n \in \N$ whenever $n_0 = 0$ (i.e. the initial guess is close
enough).
While we expect \cref{thm:ddrate} to still be better than \cite[Theorems 3.1, 3.6]{ChaTaiWanYan} for $n_0 > 0$, a complete
comparison in that case seems to be more involved and remains to be done.

\section{Surrogate Technique}\label{sec:surrogate-technique}
\label{sec:surrogate}

A surrogate iteration substitutes minimization of one functional for
minimization of different, simpler functionals at the cost of an additional
iterative process.
In particular one can substitute the minimization problem $\inf_{p\in K}
\tfrac{1}{2} \|\LambdaStar p - T^* g\|_{B^{-1}}^2$ by the iteration
\begin{align*}
  \inf_{p^{n+1}\in K} &\tfrac{1}{2} \|\LambdaStar p^{n+1} - f^n\|_W^2,
  \qquad f^n = \LambdaStar p^n - \tfrac{1}{\tau} B^{-1}(\LambdaStar p^n - T^*g),
\end{align*}
producing iterates $(p^n)_{n \in \N}$ for some initialization $p^0 \in V$ that
converge to the same minimizer, provided $\tau \in (\|B^{-1}\|, \infty)$.
Though its properties have been studied extensively in e.g.\
\cite{mairal2013optimization}, we will analyze it as a nested subalgorithm of
our decomposition scheme for approximate minimization following the
notion from \cref{def:approxmin}.
The main motivation for the surrogate technique in our case is to rid the local
problems from the dependence on the potentially costly operator $B^{-1}$.

To that end for $n\in\N_0$ we introduce an auxiliary functional $\D_i^{s,n}$ defined as
\begin{align*}
  \D_i^{s,n}(v_i, w_i) : = \D(p_{i-1}^n + (v_i - \theta_i p^n)) +
    \tfrac{1}{2} \|\LambdaStar (v_i - w_i)\|_{\tau I - B^{-1}}^2
\end{align*}
with $\tau > \|B^{-1}\|$ for $v_i,w_i\in \theta_i K$ and $i=1,\ldots,M$, whereby $\|u\|_{\tau I - B^{-1}}^2:= \<u,(\tau I - B^{-1})u\>_W$ for $u\in W$.

\begin{algorithm}
  \caption{Surrogate approximation} \label{alg:ddsur}
  \begin{algorithmic}[1]
    \Require{$N_{\text{sur}} \in \N$,
      $n \in \N_0$, $i \in \{1, \dotsc, M\}$,
      $p^n \in K$, $p_{i-1}^n \in K$}
    \Ensure{$\tilde v_i^n \in \theta_i K$}
    \State{$v_i^{n,0} = \theta_i p_{i-1}^n$}
    \For{$\ell = 0,1,\ldots, N_{\text{sur}}-1$}
      \State{$v_i^{n,\ell+1} \in \argmin_{v_i\in \theta_i K} \D_i^{s,n}(v_i, v_i^{n,\ell})$}
    \EndFor
    \State{$\tilde v_i^n = v_i^{n,N_{\text{sur}}}$}
  \end{algorithmic}
\end{algorithm}

We note that the subproblems in \cref{alg:ddsur} can be written as
\begin{align*}
  &\inf_{v_i \in \theta_i K} \D_i^{s,n}(v_i, v_i^{n,\ell}) \\
  \iff &\inf_{v_i \in \theta_i K}
    \tfrac{1}{2} \|\LambdaStar (p_{i-1}^n + (v_i - \theta_i p^n)) - T^* g\|_{B^{-1}}^2 +
    \tfrac{1}{2} \|\LambdaStar(v_i - v_i^{n,\ell})\|_{\tau I - B^{-1}}^2 \\
  \iff & \inf_{v_i \in \theta_i K} \tfrac{1}{2} \|\LambdaStar v_i - f_i^n\|_W^2,
\end{align*}
where $f_i^n = \LambdaStar v_i^{n,\ell} - \tfrac{1}{\tau}
B^{-1}(\LambdaStar(p_{i-1}^n + (v_i^{n,\ell} - \theta_i p^n)) - T^* g)$.
The dependence on the operator $B^{-1}$ has thereby been
moved into the preparation of fixed data $f_i^n$ for every subproblem, while
the subproblem itself for fixed $f_i^n$ is independent of $B^{-1}$.

\Cref{alg:ddsur} produces approximations $\tilde{v}_i^n$ to be used in
\cref{alg:ddpar,alg:ddseq}.
Following ideas from \cite[Proposition 2.2]{mairal2013optimization} we show below, that the surrogate approximation converges linearly and any
fixed number of surrogate iterations $N_{\text{sur}}$ is enough to receive the
convergence rate from \cref{thm:ddrate} for the resulting combined algorithm.

\begin{lemma} \label{lem:surrogate_quadratic_growth}
  Using notation and assumptions from \cref{alg:ddsur} the functional $\D_i^n: V_i \to \R$,
  \begin{align*}
    \D_i^n(v) := \D(p_{i-1}^n-\theta_i p^n + v),
  \end{align*}
  has quadratic growth in the sense that
  \begin{align*}
    \D_i^n(v) - \D_i^n(\hat v)
    \ge \tfrac{1}{2\|\tau I - B^{-1}\|\|B\|} \|\LambdaStar(v - \hat v)\|_{\tau I - B^{-1}}^2
  \end{align*}
  for any minimizer $\hat v \in \theta_i K$ of $\D_i^n$.
  \begin{proof}
    Using \cref{Lemma22} and optimality of $\hat v \in \theta_i K$ we see that
    \begin{align*}
      \D_i^n(v) - \D_i^n(\hat v)
      &= \<\D'(p_{i-1}^n + (\hat v - \theta_i p^n)), v - \hat v_i\>
        + \tfrac{1}{2} \|v - \hat v\|_{*}^2 \\
      &\ge \tfrac{1}{2} \|\LambdaStar (v - \hat v)\|_{B^{-1}}^2.
    \end{align*}
    Further noting that $\tau I - B^{-1}$ is positive definite, since $\tau >
    \|B^{-1}\|$,
    \begin{align*}
      \|\LambdaStar (v - \hat v)\|_{\tau I - B^{-1}}^2
      \le \|\tau I - B^{-1}\| \|\LambdaStar (v - \hat v)\|_W^2
      \le \|\tau I - B^{-1}\|\|B\| \|\LambdaStar (v - \hat v)\|_{B^{-1}}^2.
    \end{align*}
    Combining both inequalities yields the statement.
  \end{proof}
\end{lemma}

\begin{proposition} \label{prop:surrogate_step}
  Using notation and assumptions from \cref{alg:ddsur} and \cref{lem:surrogate_quadratic_growth} the surrogate iterates $(v_i^{n,\ell})_\ell$ satisfy
  \begin{align*}
    \D_i^n(v_i^{n,\ell}) - \D_i^n(v_i^{n,\ell+1})
    &\ge \eta (\D_i^n(v_i^{n,\ell}) - \D_i^n(\hat v_i^{n}))
  \end{align*}
  for all $\ell\in \N$ and for any minimizer $\hat v_i^n \in \theta_i K$ of $\D_i^n$, $i\in\{1,\ldots, n\}$, $ n\in\N_0$,
  while $\eta \in (0,1)$ is given by
  \begin{align*}
    \eta = \begin{cases}
      \tfrac{1}{4\|\tau I - B^{-1}\|\|B\|} & \text{if $\|\tau I - B^{-1}\|\|B\|
      \ge \tfrac{1}{2}$} \\
      1 - \|\tau I - B^{-1}\|\|B\| & \text{else}.
    \end{cases}
  \end{align*}
  \begin{proof}
    Since $\D_i^{s,n}(v,w) = \D_i^n(v) + \tfrac{1}{2} \|\LambdaStar(w-v)\|_{\tau I - B^{-1}}^2$ we have
    \begin{align*}
      &\D_i^n(v_i^{n,\ell+1}) + \tfrac{1}{2} \|\LambdaStar(v_i^{n,\ell} - v_i^{n,\ell+1})\|_{\tau I - B^{-1}}^2 \\
      &\quad= \D_i^{s,n}(v_i^{n,\ell+1}, v_i^{n,\ell}) \\
      &\quad= \min_{v_i \in \theta_i K} \D_i^n(v_i) + \tfrac{1}{2} \|\LambdaStar(v_i^{n,\ell} - v_i)\|_{\tau I - B^{-1}}^2\\
      &\quad\le \min_{\mu \in [0,1]} \D_i^n((1-\mu) v_i^{n,\ell} + \mu \hat
      v_i^n) + \tfrac{\mu^2}{2} \|\LambdaStar(v_i^{n,\ell} - \hat v_i^n)\|_{\tau I - B^{-1}}^2 \\
      &\quad\le \min_{\mu \in [0,1]} (1-\mu) \D_i^n(v_i^{n,\ell}) + \mu
      \D_i^n(\hat v_i^n) +  \tfrac{\mu^2}{2} \|\LambdaStar(v_i^{n,\ell} - \hat
      v_i^n)\|_{\tau I - B^{-1}}^2,
    \end{align*}
    where we searched for the minimum along the line $v_i = (1-\mu)v_{i}^{n,\ell}
    + \mu \hat v_i^n \in \theta_i K$, $\mu \in [0, 1]$, and used convexity afterwards.
    After reordering we use the quadratic growth property from \cref{lem:surrogate_quadratic_growth} to see that
    \begin{align*}
      &\D_i^n(v_i^{n,\ell}) - \D_i^n(v_i^{n,\ell+1}) - \tfrac{1}{2} \|\LambdaStar(v_i^{n,\ell} - v_i^{n,\ell+1})\|_{\tau I - B^{-1}}^2 \\
      &\quad\ge \max_{\mu \in [0,1]} \mu (\D_i^n(v_i^{n,\ell}) - \D_i^n(\hat
      v_i^n)) - \tfrac{\mu^2}{2} \|\LambdaStar(v_i^{n,\ell} - \hat v_i^n)\|_{\tau I - B^{-1}}^2 \\
      &\quad\ge \max_{\mu \in [0,1]} \big(\mu - \mu^2 \|\tau I - B^{-1}\|\|B\|\big) (\D_i^n(v_i^{n,\ell}) - \D_i^n(\hat v_i^n)).
    \end{align*}
    Discarding the last term on the left-hand side and evaluating the maximum
    optimally at $\mu = \min\{1, \frac{1}{2\|\tau I - B^{-1}\| \|B\|} \} \in (0,1]$ yields
    \begin{align*}
      \D_i^n(v_i^{n,\ell}) - \D_i^n(v_i^{n,\ell+1})
      &\ge \eta (\D_i^n(v_i^{n,\ell}) - \D_i^n(\hat v_i^{n}))
    \end{align*}
    where $\eta \in (0,1)$ is given by
    \begin{align*}
      \eta &= \begin{cases}
	\tfrac{1}{4\|\tau I - B^{-1}\|\|B\|} &
          \text{if $\|\tau I - B^{-1}\|\|B\| \ge \tfrac{1}{2}$} \\
	1 - \|\tau I - B^{-1}\|\|B\| &
          \text{else}.
      \end{cases}
    \end{align*}
    ~
  \end{proof}
\end{proposition}

\Cref{prop:surrogate_step} is sharp in the sense that for trivial $B^{-1} = I$
and minimizing $1 < \tau \to 1$, we recover the optimal factor $\eta \to 1$.

\begin{lemma} \label{lem:surrogate_approx}
  The surrogate iterates $(v_i^{n,\ell})_\ell$ from
  \cref{alg:ddsur} yield approximate solutions to the subproblems in the
  sense that
  \begin{align*}
    \D_i^n(v_i^{n,0}) - \D_i^n(v_i^{n,\ell})
    \ge \big(1 - (1-\eta)^\ell\big) (\D_i^n(v_i^{n,0}) - \D_i^n(\hat v_i^n))
  \end{align*}
  for any minimizer $\hat v_i^n \in \theta_i K$ of $\D_i^n$, $i \in \{1 \dotsc,
  M\}$, $n \in \N_0$ and $\eta \in (0,1)$
  defined as in \cref{prop:surrogate_step}.
  \begin{proof}
    Elementary calculation using \cref{prop:surrogate_step} yields a linear energy decrease
    \begin{align*}
      \D_i^n(v_i^{n,\ell+1}) - \D_i^n(\hat v_i^n)
      &= -(\D_i^n(v_i^{n,\ell}) - \D_i^n(v_i^{n,\ell+1})) + \D_i^n(v_i^{n,\ell}) - \D_i^n(\hat v_i^n) \\
      &\le -\eta(\D_i^n(v_i^{n,\ell}) - \D_i^n(\hat v_i^n)) + \D_i^n(v_i^{n,\ell}) - \D_i^n(\hat v_i^n) \\
      &= (1-\eta)(\D_i^n(v_i^{n,\ell}) - \D_i^n(\hat v_i^n))
    \end{align*}
    which we use to find
    \begin{align*}
      \D_i^n(v_i^{n,0}) - \D_i^n(v_i^{n,\ell})
      &= \D_i^n(v_i^{n,0}) - \D_i^n(\hat v_i^n) - (\D_i^n(v_i^{n,\ell}) - \D_i^n(\hat v_i^n)) \\
      &\ge \D_i^n(v_i^{n,0}) - \D_i^n(\hat v_i^n) -(1-\eta)^\ell(\D_i^n(v_i^{n,0}) - \D_i^n(\hat v_i^n)) \\
      &\ge (1 -(1-\eta)^\ell)(\D_i^n(v_i^{n,0}) - \D_i^n(\hat v_i^n)).
    \end{align*}
  \end{proof}
\end{lemma}

Finally, combining \cref{thm:ddrate} with \cref{lem:surrogate_approx} then
immediately yields the following corollary.

\begin{corollary} \label{cor:ddrate-surrogate}
  \Cref{alg:ddpar,alg:ddseq} with subproblems solved using
  \cref{alg:ddsur} converge in the sense that $\D(p^n) \to \D(\hat p)$.
  Furthermore
  \begin{align*}
    \D(p^n) - \D(\hat p)
    &\le \begin{cases}
      (1 - \tfrac{\rho\sigma}{2\alpha})^n \big(\D(p^{0}) - \D(\hat p)\big) & \text{if $n \le n_0$} \\
      \tfrac{2\Phi^2}{\rho \sigma}\alpha^2 (n - n_0 +1)^{-1} & \text{if $n \ge n_0$,}
    \end{cases}
  \end{align*}
  where $\alpha := 1 + M\sigma\sqrt{2-\rho+2\sqrt{1-\rho}}$ for \cref{alg:ddseq}
  and $\alpha := 1$ for \cref{alg:ddpar},
  $\Phi := \sqrt{\|B^{-1}\|} \NormLambda C_\theta R_{\hat p}$,
  $n_0 := \min \{n \in \N_0 : \D(p^{n}) - \D(\hat p) < \Phi^2 \alpha\}$
  and
  \begin{align*}
    \rho = (1-(1-\eta)^{N_{\text{sur}}}), \qquad
    \eta = \begin{cases}
      \tfrac{1}{4\|\tau I - B^{-1}\|\|B\|} & \text{if $\|\tau I - B^{-1}\|\|B\|
      \ge \tfrac{1}{2}$} \\
      1 - \|\tau I - B^{-1}\|\|B\| & \text{else}.
    \end{cases}
  \end{align*}
  for any fixed number of inner surrogate iterations $N_{\text{sur}} \in \N$.
\end{corollary}

\begin{remark}
In \cref{alg:ddsur} we specify a fixed number of surrogate iterations over all subdomain problems, i.e.\ $N_{sur}$ is the same in each subdomain. However, one may indeed make $N_{sur}$ dependent on $\Omega_i$ leading to $N_{sur,i}$ for $i=1,\ldots,M$. Note that this does not change the statements in \cref{prop:surrogate_step,lem:surrogate_approx} as these estimates only concern the subproblems separately. Moreover, the estimate in \cref{lem:surrogate_approx} is the weaker the smaller $\ell\in\N$ is, since $1-(1-\eta)^{a} > 1-(1-\eta)^{b}$ for $a>b$ as $1-\eta \in (0,1)$. Together with this observation we obtain that in the case of subdomain dependent inner surrogate iterations \cref{cor:ddrate-surrogate} then holds with replacing $N_{sur}$ by $\min_{i\in\{1,\ldots,M\}}N_{sur,i}$, i.e.\ by the minimal number of inner surrogate iterations over all subdomains.
\end{remark}

\begin{remark}\label{rem:general:setting} The above presented algorithms and its analysis is not restricted to problem \eqref{eq:model-motivation-predual} and also holds for more general problems of the following type 
\begin{equation} \label{eq:model}
  \inf_{p \in K} \Big\{\tilde{\D}(p) := \tfrac{1}{2} \|\LambdaStar p - f\|_{B^{-1}}^2 \Big\},
\end{equation}
where $\LambdaStar: V \to W$ is a bounded linear operator, $V, W$ are  real Hilbert spaces, 
$B^{-1}: W \to W$ a positive definite self-adjoint bounded linear operator, $K
\subset V$ a closed convex set, $f \in W$, and $\|q\|_{B^{-1}}^2 := \<B^{-1}q, q\>_W$ for $q \in W$. Assuming coercivity of $\tilde{\D}$ ensures the existence of a solution of \eqref{eq:model}. 

In particular we obtain the same convergence order results for \eqref{eq:model} as for \eqref{eq:model-motivation-predual}, i.e. \cref{thm:ddrate} and \cref{cor:ddrate-surrogate} also hold in case of \eqref{eq:model}.
\end{remark}

\section{Semi-Implicit Dual Algorithm}\label{sec:semi-implicit-algorithm}

Solution strategies for solving \eqref{eq:model-motivation-predual} are especially relevant in our decomposition setting of
\cref{alg:ddpar,alg:ddseq}, since we have to solve subproblems
\eqref{eq:algpdd-subproblem} of the same general form. 
One specific such algorithm is the semi-implicit Lagrange multiplier method due
to Chambolle \cite{Chambolle:04} which solves
\eqref{eq:model-motivation-predual} for the special case $B = I$.
While \cite{Chambolle:04} uses finite differences, we present the
algorithm in a Hilbert space setting and for more general $B$.

As in \eqref{eq:model-motivation-predual} let
$K := \{{p} \in V: |{p}|_F \le \lambda\}$
denote the set of feasible
dual variables.
Similar to \cite{Chambolle:04}
there exists a Lagrange multiplier $\mu \in L^\infty(\Omega)$
corresponding to the constraint in $K$,
c.f.\ \cite[Theorem 1.6]{ito2008lagrange}, such that $p \in V$ is a solution of
\eqref{eq:model-motivation-predual} if and only if
\begin{equation} \label{eq:alg_dual_kkt}
  0 = \LambdaOp B^{-1}(\LambdaStar p - T^*g) + \mu p
\end{equation}
with $\mu \ge 0$, $|p|_F \le 1$ and $\tfrac{\mu}{2}(|p|_F^2 - \lambda^2) = 0$ holds.
Here $\mu p$ is to be understood as pointwise multiplication.

Recall that $\lambda > 0$.
Observing that in a pointwise sense $\mu = 0$ implies $\xi :=
\LambdaOp B^{-1}(\LambdaStar p - T^*g) = 0$ and $\mu > 0$ implies $|p|_F =
\lambda$ almost everywhere,
we deduce from condition \eqref{eq:alg_dual_kkt}, that in either case $\mu =
\tfrac{|\xi|_F}{\lambda}$.
Thus \eqref{eq:alg_dual_kkt} becomes
\begin{align*}
  0 = \xi + \tfrac{|\xi|_F}{\lambda} p.
\end{align*}
The semi-implicit iterative method then uses for some starting value $p^0
\in K$ and stepsize $\tau > 0$ iterates $(p^n)_{n\in \N_0} \subset V$
satisfying
\begin{equation} \label{eq:alg_dual_iterate}
  p^{n+1}
  = p^n - \tau (\xi^n + \tfrac{|\xi^n|_F}{\lambda} p^{n+1}),
\end{equation}
where $\xi^n := \LambdaOp B^{-1}(\LambdaStar p^n - T^*g)$, $n \in \N_0$.
Solving \eqref{eq:alg_dual_iterate} for $p^{n+1}$ then yields \cref{alg:chambolle}.

\begin{algorithm}
  \caption{Semi-implicit dual multiplier method \cite{Chambolle:04}} \label{alg:chambolle}
  \begin{algorithmic}[1]
    \Require{$p^0 \in K$ and $\tau \in (0, \tfrac{1}{\|\LambdaOp B^{-1}\LambdaStar\|})$}
    \For{$n=0,1,2,\ldots$}
      \State{$\xi^n = \LambdaOp B^{-1}(\LambdaStar p^n - T^*g)$}
      \State{$p^{n+1} = \displaystyle
        \lambda \frac{p^n - \tau \xi^n}{\lambda + \tau |\xi^n|_F}$}
    \EndFor
  \end{algorithmic}
\end{algorithm}

Before we prove convergence of \cref{alg:chambolle} in the upcoming
\cref{thm:chambolle_convergence}, let us first make some comments on
\cref{alg:chambolle}.

\begin{remark} \label{rem:chambolle-notes}
  Regarding \cref{alg:chambolle}, take note of the following:
  \begin{itemize}
    \item
      In the trivial case $\lambda = 0$ one sets $p^{n+1} = 0$.
    \item
      If $B^{-1}$ is a local operator, the computation of $\xi^n$ and $p^{n+1}$
      are both local.
      They can therefore be merged together and carried out in parallel over
      the whole domain.
    \item
      One may solve the decomposition subproblems
      \eqref{eq:algpdd-subproblem} by replacing $K$ with $\theta_i K$ and
      $\lambda$ with the pointwise function $\theta_i \lambda$.
    \item
      A more explicit bound for the maximum stepsize $\tau$ to still
      analytically ensure convergence is given by
      \begin{align*}
        \|\LambdaOp B^{-1} \LambdaStar\|
        \le \|\nabla\|^2 \|B^{-1}\|.
      \end{align*}
      For finite differences as used in this paper (see
      \cref{def:finite-differences} below) we have \cite{Chambolle:04}
      \begin{align*}
        \|\nabla_h\|^2 \le 8.
      \end{align*}
  \end{itemize}
\end{remark}

\begin{theorem}[c.f.\ {\cite[Theorem 3.1]{Chambolle:04}}]
    \label{thm:chambolle_convergence}
  Let $p^0 \in K$. Then \cref{alg:chambolle} generates a sequence $(p^n)_{n\in\N_0} \subset K$ such that $\D(p^n) \to \D(\hat p)$ for $n\to\infty$ if $V$ is finite dimensional, where $\hat{p} \in K$ is a minimizer of \eqref{eq:model-motivation-predual}.
  
  \begin{proof}
    We follow along the lines of \cite[Theorem 3.1]{Chambolle:04}.
    Notice that $|p^0|_F \le \lambda$ and thus inductively
    \begin{align*}
      |p^{n+1}|_F
      \le \lambda \tfrac{|p^n|_F + \tau |\xi^n|_F}{\lambda + \tau |\xi^n|_F}
      \le \lambda,
    \end{align*}
    i.e.\ $p^n \in K$ for all $n \in \N$.
    Let $F: V \to V$ denote the iteration function of \cref{alg:chambolle}, such that
    $p^{n+1} = F(p^n)$, $n \in \N_0$.
    Any fixed point of $F$ or equivalently of
    \eqref{eq:alg_dual_iterate} satisfies the
    stationary point condition \eqref{eq:alg_dual_kkt} per construction and, since $\D$
    is convex, will be a minimizer of \eqref{eq:model-motivation-predual}.

    Denote $\eta^n := \tfrac{1}{\tau}(p^n - p^{n+1})$ and bound the energy
    difference
    \begin{align*}
      \D(p^n) - \D(p^{n+1})
      &= -\tfrac{1}{2} \|p^n - p^{n+1}\|_*^2 + \<\D'(p^n), p^n - p^{n+1}\>
        \tag{\cref{Lemma22} (i)} \\
      &= -\tfrac{\tau^2}{2} \|\eta^n\|_*^2
        + \tau\<\xi^n, \eta^n\> \\
      &= \tfrac{\tau}{2}(\|\eta^n\|^2 - \tau\|\eta^n\|_*^2)
        + \tau \<\xi^n - \tfrac{1}{2}\eta^n, \eta^n\> \\
      &= \tfrac{\tau}{2}(\|\eta^n\|^2 - \tau\|\eta^n\|_*^2)
        + \tfrac{\tau}{2} \<\xi^n - \tfrac{|\xi^n|_F}{\lambda} p^{n+1},
          \xi^n + \tfrac{|\xi^n|_F}{\lambda} p^{n+1}\>
          \tag{applying \eqref{eq:alg_dual_iterate}} \\
      &= \tfrac{\tau}{2}(\|\eta^n\|^2 - \tau\<\LambdaOp B \LambdaStar \eta^n, \eta^n\>)
        + \tfrac{\tau}{2}(\|\xi^n\|^2 - \|\tfrac{|\xi^n|_F}{\lambda} p^{n+1}\|^2) \\
      &\ge \tfrac{\tau}{2}\big(1 - \tau \big\|\LambdaOp B^{-1}
        \LambdaStar\big\|\big)\|\eta^n\|^2
        + \tfrac{\tau}{2}\|\xi^n\|^2
          \big(1 - \big\|\tfrac{|p^{n+1}|_F^2}{\lambda^2}\big\|_{L^\infty}\big) \\
      &\ge \tfrac{1}{2\tau}(1 - \tau \|\LambdaOp B^{-1} \LambdaStar\|)
          \|p^n - p^{n+1}\|^2.
    \end{align*}
    We see that as long as $\tau < \|\LambdaOp B^{-1} \LambdaStar\|^{-1}$,
    the sequence $(\D(p^n))_{n\in\N_0}$ is non-increasing and thus, since it is
    non-negative, also convergent.
    The feasible set $K \subset V$ is closed and bounded, see
    \cite[Lemma 3.4]{ours2022semismooth}, and compact since $V$ is finite
    dimensional.
    Consequently there exists a convergent subsequence $(q^n)_{n\in\N} \subset
    (p^n)_{n\in\N} \subset K$, $q^n \to q \in K$ and with continuity of $F$
    we get $F(q^n) \to F(q)$.
    Using the estimate above and the convergence of energies we see that for
    some $c > 0$ we have
    $c\|q^n - F(q^n)\|^2 \le \D(q^n) - \D(F(q^n)) \to 0$ and therefore the limit
    needs to be a fixed point, $q = F(q)$, and thus a minimizer of
    \eqref{eq:model-motivation-predual}.
  \end{proof}
\end{theorem}

We note that \cref{thm:chambolle_convergence} guarantees the convergence to a
minimal dual energy, which allows us to reconstruct the optimal primal solution
due to \cref{prop:dual-energy-strong-convexity}.
It does, however, not guarantee convergence of the dual iterates
$(p^n)_{n\in\N}$ themselves.

\section{Numerical Experiments}\label{sec:numerics}

Let for $a,b \in \Z^d$ the discrete domain be given by
\begin{align*}
  \Omega_{h}
  := \Omega_{h,[a,b]}
  := \big\{ x = (x_1, \dotsc, x_d) \in \Z^d : a \le x \le b \big\}\subset \Z^d.
\end{align*}

Computer images given by an array $A \in [0, 1]^{n_1 \times \dotsb
\times n_d}$, $n = (n_1, \dotsc, n_d) \in \N^d$ of intensity values between $0$
(black) and $1$ (white) are then mapped to a discrete function $u_h:
\Omega_{h, [1,n]} \to \R$ by defining $u_h(x) :=
A_x$.

\begin{definition}[Finite Difference Operators] \label{def:finite-differences}
  For $u_h: \Omega_{h} \to \R^m$ and $p_h = (p_{h,1}, \dotsc, p_{h,d}):
  \Omega_{h} \to \R^{d\times m}$ let
  forward differences $\partial_{h,k}^+: \Omega_{h} \to \R^m$ and backward differences
  $\partial_{h,k}^-: \Omega_{h} \to \R^m$ be
  given by
  \begin{align*}
    \partial_{h,k}^+ u_h(x) &:= \begin{cases}
      0 & \text{if $x_k = b_k$}, \\
      u_h(x+he^k) - u_h(x) & \text{else},
    \end{cases}, \\
    \partial_{h,k}^- u_h(x) &:= \begin{cases}
      u_h(x) & \text{if $x_k = a_k$}, \\
      -u_h(x-he^k) & \text{if $x_k = b_k$}, \\
      u_h(x) - u_h(x-he^k) & \text{else},
    \end{cases}
  \end{align*}
  where $e^k \in \N^d$ denotes the $k$-th unit vector, $k = 1, \dotsc, d$.
  The discrete gradient $\nabla_h u_h: \Omega_{h} \to \R^{d\times m}$ and
  discrete divergence $\Div_h p_h: \Omega_{h} \to \R^m$ are then defined as
  \begin{align*}
    \nabla_h u_h := (\partial_{h,k}^+ u_h)_{k=1}^d,
    \qquad
    \Div_h p_h := \sum_{k=1}^d \partial_{h,k}^- p_{h,k}.
  \end{align*}
\end{definition}

For a given discrete overlap $r \in \N$ and a desired number of domains $M \in \N$
we first define a discrete covering of $\Omega_h$ in dimension $d = 1$.
Let $s := |b - a|$ be the diameter of $\Omega_h$, i.e. its length.
Define $M$ approximately equal integer
sublengths given recursively by 
\begin{align*}
  a_i := \Big\lfloor \frac{s + (M-1)r - \sum_{j=1}^{i-1} (a_j - r)}{M - (i-1)} \Big\rfloor,
  \qquad i = 1, \dotsc, M.
\end{align*}
These give rise to the subdomains
\begin{align*}
  \Omega_{h,i} := \{b_i, b_i + 1, \dotsc, b_i + a_i\},
  \qquad
  b_i := \sum_{j=1}^{i-1} (a_j - r),
  \qquad
  i = 1, \dotsc, M
\end{align*}
of diameter $a_i$ and the partition functions $\theta_{h,i}: \Omega_h \to [0,1]$ by
\begin{align*}
  \theta_{h,i}(x) := \min\big\{1, \tfrac{1}{r} \operatorname{dist}(x, [0, s] \setminus [b_i, b_i + a_i]) \big\},
\end{align*}
where $\operatorname{dist}$ is the (Euclidean) distance function.
The above construction in one dimension yields $M$ discrete subdomains
$\Omega_{h,i}$ and a corresponding partition of unity $\theta_{h,i}$ for a
discrete domain $\Omega_h$ of any size provided $M$ and $r$ are chosen such that $a_i \ge 2r$.

Higher dimensions $d > 1$ are realized through a standard tensor-product formulation based on the above construction, yielding $M = \prod_{k=1}^d M_k$ subdomains with overlaps $r = (r_1, \dotsc, r_d)$.

In all our decomposition examples we use \cref{alg:chambolle} as a subproblem solver
if not specified otherwise and choose its stepsize
$\tau = \tfrac{1}{8\|B^{-1}\|}$ in accordance with \cref{rem:chambolle-notes}.

The source code for all following numerical examples has been made available
under a permissive license \cite{DualTVDDJl}.

\subsection{Convergence}

\InputData{data1/fd/convergence/denoising/data.tex}[convergence/denoising]
\InputData{data1/fd/convergence/inpainting/data.tex}[convergence/inpainting]
\InputData{data1/fd/convergence/opticalflow/data.tex}[convergence/opticalflow]

\DeclareDocumentCommand{\Algplots}{m}{%
  \begin{subfigure}[t]{0.9\columnwidth}%
    \begin{tikzpicture}%
      \begin{loglogaxis}[
        width=\columnwidth,
        height=5cm,
        xlabel={iteration $k$},
        ylabel={$\D(p^k) - \D_{\text{min}}$},
        legend pos=south west,
        legend cell align=left,
        unbounded coords=jump,
        ]
        \addplot table [
            col sep=comma,
            x=k,y=glob_energy,
          ] {data1/fd/convergence/#1/energies.csv};
        \addlegendentry{global};
        \addplot table [
            col sep=comma,
            x=k,y=ddseq_energy,
          ] {data1/fd/convergence/#1/energies.csv};
        \addlegendentry{dd sequential};
        \addplot table [
            col sep=comma,
            x=k,y=ddpar_energy,
          ] {data1/fd/convergence/#1/energies.csv};
        \addlegendentry{dd parallel};
      \end{loglogaxis}
    \end{tikzpicture}
    \caption{energy}
  \end{subfigure}
}

We numerically verify the theoretical sublinear convergence properties of
\cref{alg:ddpar} and \cref{alg:ddseq} due to \cref{thm:ddrate} for three different
applications, i.e.\ image denoising, image inpainting and estimating the optical flow.

For each application described below we compare \cref{alg:chambolle} on the
global, non-decomposed problem and the decomposition
\cref{alg:ddpar,alg:ddseq}.
For the global algorithm we abort after $10^6$ iterations, while for the
decomposition algorithms with \cref{alg:chambolle} as a subalgorithm solver,
we abort after $\Data{convergence/denoising/maxiters}$ outer iterations and
each subalgorithm after $\Data{convergence/denoising/ninner}$ inner iterations.
For a fair comparison we denote with $k \in \N$ the outer iterations of
\cref{alg:ddpar,alg:ddseq} and the iterations of the global algorithm inversely
scaled by the number of inner iterations of the decomposition algorithms, that is
\[
k=\begin{cases}
n & \text{if \cref{alg:ddpar} or \cref{alg:ddseq} is used};\\
\frac{n}{100} &\text{if \cref{alg:chambolle} is used as global algorithm}.
\end{cases}
\]
All three algorithms are initialized using $p^0 = 0$.
For both \cref{alg:ddpar,alg:ddseq} in each outer iteration $n$ the subalgorithm on
the $i$-th subdomain is initialized with the current subdomain view $\theta_i
p^n_{i-1} \in \theta_i K$.

In each case we downsample input images to a small size of
$\Data{convergence/denoising/width} \times \Data{convergence/denoising/height}$ pixels and
decompose the domain into $M = 2 \cdot 2$ subdomains with
an overlap of $r_1 = r_2 = 5$ pixels in order to make a very high number of iterations
timely feasible.
For \cref{alg:ddseq} and \cref{alg:ddpar} we use the largest allowable value of
$\sigma$, i.e.\ $\sigma = 1$ and $\sigma = \tfrac{1}{4}$ respectively.

The three applications are realized by making use of \cref{prop:duality}
and setting data $g$, operator $T$ and model parameters $\lambda, \beta$
therein as follows.
\begin{description}
  \item[Denoising]
    We start with a ground truth image $\tilde g$ and
    generate an artificially
    noisy input $g = \tilde g + \eta$, where $\eta$ denotes zero-mean additive
    Gaussian noise with variance $\Data{convergence/denoising/noise_sigma}$.
    Setting $T=I$, and choosing model
    parameters
    $\lambda = \Data{convergence/denoising/lambda}$,
    $\beta = \Data{convergence/denoising/beta}$ we apply the respective
    algorithm to obtain the denoised output $u$.
  \item[Inpainting]
    Starting with a ground truth image $\tilde g$ we artificially mask each
    pixel with probability $\frac{1}{2}$ by setting its value to $0$ (black) to
    receive a corrupted input image $g$.
    Denoting by $A \subseteq \Omega$ the masked area we set
    $T = \mathbf{1}_{\Omega\setminus A}$ where $\mathbf{1}_{\Omega\setminus A}$
    is the indicator function on $\Omega\setminus A$
    while the model parameters are chosen to be
    $\lambda = \Data{convergence/inpainting/lambda}$,
    $\beta = \Data{convergence/inpainting/beta}$.
  \item[Optical flow estimation]
    Given two greyscale images
    $g_0, g_1: \Omega \to [0,1]$ we estimate their
    vector-valued optical flow displacement field $u$
    by setting the difference $g = g_0 - g_1$ as input data and
    $T$ using $Tu = \nabla g_1 \cdot u$.
    This formulation is the linear approximation of the brightness
    constancy constraint suitable for small displacements
    and may be found in \cite[(5.81)]{AubKor}.
    Model parameters are set to
    $\lambda = \Data{convergence/opticalflow/lambda}$,
    $\beta = \Data{convergence/opticalflow/beta}$.
    We visualize the optical flow field $u$ and the benchmark-provided
    ground truth as a color-coded image following
    \cite{BaScLeRoBlSz:11}.
\end{description}

For each of the applications we denote by $\D_{\text{min}}$ the minimum energy
obtained by running the global \cref{alg:chambolle} for a maximum of $10^7$
iterations.
For denoising, inpainting and optical flow estimation we determined
$\D_{\text{min}} \approx \Data{convergence/denoising/energymin}$,
$\D_{\text{min}} \approx \Data{convergence/inpainting/energymin}$ and
$\D_{\text{min}} \approx \Data{convergence/opticalflow/energymin}$ respectively.
The energy values over all iteration for the three algorithms and each
application are plotted in
\cref{fig:convergence_denoising,fig:convergence_inpainting,fig:convergence_opticalflow}
together with respective input and output images.

\begin{figure}%
  \centering%
  \Algplots{denoising}
  \begin{subfigure}[t]{0.3\columnwidth}
    \centering
    \includegraphics[width=\textwidth]{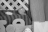}
    \caption{ground truth image}
  \end{subfigure}
  \begin{subfigure}[t]{0.3\columnwidth}
    \centering
    \includegraphics[width=\textwidth]{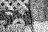}
    \caption{noisy input image}
  \end{subfigure}
  \\
  \begin{subfigure}[t]{0.3\columnwidth}
    \centering
    \includegraphics[width=\textwidth]{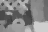}
    \caption{denoised output image, global}
  \end{subfigure}
  \begin{subfigure}[t]{0.3\columnwidth}
    \centering
    \includegraphics[width=\textwidth]{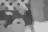}
    \caption{denoised output image, dd sequential}
  \end{subfigure}
  \begin{subfigure}[t]{0.3\columnwidth}
    \centering
    \includegraphics[width=\textwidth]{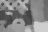}
    \caption{denoised output image, dd parallel}
  \end{subfigure}
  \caption{denoising: convergence of energy and results}
  \label{fig:convergence_denoising}
\end{figure}

\begin{figure}%
  \centering%
  \Algplots{inpainting}
  \begin{subfigure}[t]{0.3\columnwidth}
    \centering
    \includegraphics[width=\textwidth]{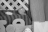}
    \caption{ground truth image}
    \label{fig:rubberwhale1}
  \end{subfigure}
  \begin{subfigure}[t]{0.3\columnwidth}
    \centering
    \includegraphics[width=\textwidth]{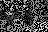}
    \caption{corrupted input image}
    \label{fig:rubberwhale_region}
  \end{subfigure}
  \\
  \begin{subfigure}[t]{0.3\columnwidth}
    \centering
    \includegraphics[width=\textwidth]{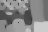}
    \caption{inpainted output image, global}
  \end{subfigure}
  \begin{subfigure}[t]{0.3\columnwidth}
    \centering
    \includegraphics[width=\textwidth]{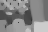}
    \caption{inpainted output image, dd sequential}
  \end{subfigure}
  \begin{subfigure}[t]{0.3\columnwidth}
    \centering
    \includegraphics[width=\textwidth]{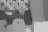}
    \caption{inpainted output image, dd parallel}
  \end{subfigure}
  \caption{inpainting: convergence of energy and results}
  \label{fig:convergence_inpainting}
\end{figure}

\begin{figure}%
  \centering%
  \Algplots{opticalflow}
  \begin{subfigure}[t]{0.3\columnwidth}
    \centering
    \includegraphics[width=\textwidth]{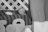}
    \caption{first image $f_0$ of image sequence}
  \end{subfigure}
  \begin{subfigure}[t]{0.3\columnwidth}
    \centering
    \includegraphics[width=\textwidth]{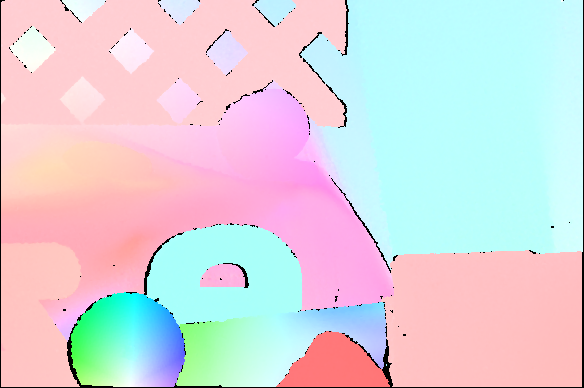}
    \caption{optical flow ground truth from \cite{BaScLeRoBlSz:11} (original resolution)}
  \end{subfigure}
  \\
  \begin{subfigure}[t]{0.3\columnwidth}
    \centering
    \includegraphics[width=\textwidth]{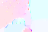}
    \caption{computed optical flow, global}
  \end{subfigure}
  \begin{subfigure}[t]{0.3\columnwidth}
    \centering
    \includegraphics[width=\textwidth]{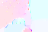}
    \caption{computed optical flow, dd sequential}
  \end{subfigure}
  \begin{subfigure}[t]{0.3\columnwidth}
    \centering
    \includegraphics[width=\textwidth]{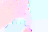}
    \caption{computed optical flow, dd parallel}
  \end{subfigure}
  \caption{optical flow: convergence of energy and results}
  \label{fig:convergence_opticalflow}
\end{figure}

We observe in \cref{fig:convergence_denoising} similar behaviour as in \cite{ChaTaiWanYan}, i.e.\ the sequential
decomposition has a slight edge on the global algorithm due to domain-overlap, while the energy curve of the parallel averaging algorithm displays a characteristic bulge in the beginning.
In \cref{fig:convergence_inpainting,fig:convergence_opticalflow} the performance difference during the iterations
between the sequential and parallel algorithm is less visible for both
inpainting and optical flow estimation.
In all cases the domain decomposition algorithms converge at a sublinear rate comparable to the respective global algorithm.

\subsection{Surrogate}

\InputData{data1/fd/surrogate/outerinner/data.tex}[surrogate/outerinner]

For local operators $B$ we compare (i) nesting the surrogate iteration
(\cref{alg:ddsur}) within domain decomposition and (ii) nesting domain
decomposition within a global surrogate iteration.
Note that for $B = I$, $\tau \to 1$ and a single surrogate iteration both of
these are identical.

We use the optical flow problem with frames of original size
$\Data{surrogate/outerinner/width} \times \Data{surrogate/outerinner/height}$
pixels
and model parameters
$\beta = \Data{surrogate/outerinner/beta}$, $\lambda = \Data{surrogate/outerinner/lambda}$.
The number of subdomains is $M = 4 \cdot 4$ with larger overlap $r_1 = r_2 = 50$ pixels
corresponding to the larger image size.
We perform for both ways of nesting $\Data{surrogate/outerinner/ninner}$
iterations of the inner algorithm and stop the whole algorithm after
$\Data{surrogate/outerinner/maxiters}$ outer iterations.
We estimate the minimal energy $\D_{\text{min}} \approx
\Data{surrogate/outerinner/energymin}$ by running \cref{alg:chambolle} for
$\Data{surrogate/outerinner/minenergy_maxiters}$ iterations.

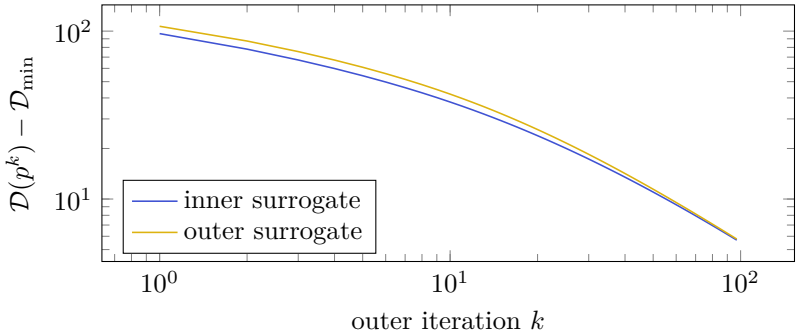
\begin{figure}
  \centering
  \begin{subfigure}[t]{0.9\columnwidth}
    \centering
    \begin{tikzpicture}
      \begin{loglogaxis}[
        width=\columnwidth,
        height=5cm,
        xlabel={outer iteration $k$},
        ylabel={$\D(p^k) - \D_{\text{min}}$},
        legend pos=south west,
        legend cell align=left,
        unbounded coords=jump,
        ]
        \addplot table [
            col sep=comma,
            x=k,y=ddseq_surrogate_energy,
          ] {data1/fd/surrogate/outerinner/energies.csv};
        \addlegendentry{inner surrogate};
        \addplot table [
            col sep=comma,
            x=k,y=surrogate_ddseq_energy,
          ] {data1/fd/surrogate/outerinner/energies.csv};
        \addlegendentry{outer surrogate};
      \end{loglogaxis}
    \end{tikzpicture}
    \caption{energy}
  \end{subfigure}
  \caption{Comparison of outer and inner surrogate, \Data{surrogate/outerinner/ninner} inner iterations per domain decomposition iteration, one single inner iteration per surrogate iteration}
  \label{fig:surrogate_outerinner}
\end{figure}

Both nestings perform similarly as can be seen in \cref{fig:surrogate_outerinner}, while nesting the surrogate iteration within the domain decomposition has a slight edge. This can be attributed to additional evaluations of $B$ in regions of overlap.

\subsection{Wavelet Transformation}

\InputData{data1/fd/global/basic/data.tex}[global/basic]

To demonstrate feasibility of our method even for global operators, we aim to
apply it to the reconstruction of corrupted wavelet coefficients.
To that end we first define the Wavelet transform $T^\infty$ in a way
convenient to us for working with arbitrarily sized images.

Denote $\Omega_s := \Omega_{h,[1,s]}$ for $s \in \N_0^d$ and
let $k = k(s) \in \N_0^d$ be such that $2k \le s \le 2k + 1$.
We define the $d$-dimensional $n$-th level discrete
Haar wavelet transform $T^n: \R^{\Omega_s} \to \R^{\Omega_s}$ recursively
by $T^0 := I$ and
for $n \ge 1$ by
\begin{align*}
  (T^n u)(\alpha \cdot k + x) := \begin{cases}
    (T^{n-1} T_0 u|_{\Omega_{2k}})(x) & \text{if $\alpha = 0$, $k \ge 1$}, \\
    (T_\alpha u|_{\Omega_{2k}})(x) & \text{if $0 \neq \alpha \le 1$, $k \ge 1$}, \\
    u(\alpha \cdot k + x) & \text{else,}
  \end{cases}
\end{align*}
for all $\alpha \cdot k + x \in \Omega_s$, where
$u: \Omega_s \to \R$, $\alpha, x \in \N^d$, $x \le k$ and the transformation
$T_\alpha: \R^{\Omega_{2k}} \to \R^{\Omega_{k}}$ on the orthant indicated by $\alpha \in \{0,1\}^d$ is given by
\begin{align*}
  (T_\alpha u)(x) := 2^{-\frac{d}{2}}
    \sum_{\substack{\beta\in\N_0^d\\\beta\le 1}}
    (-1)^{|\alpha \cdot \beta|} u\big(2(x - 1) + 1 +\beta\big)
\end{align*}
for all $x \in \N^d$, $x \le k$.
Since $T_\alpha: \R^{\Omega_{2k}} \to \R^{\Omega_k}$ halves the size and for $s
\le 1$ we have $T^n = I$ for any $n \in \N$, the operator $T^n$ becomes idempotent for large enough $n$ and we thus conveniently denote by $T^\infty := \lim_{n\to\infty} T^n$ the full wavelet transform.

We realize the application for wavelet inpainting by again making use of
\cref{prop:duality} and define data $g$, operator $T$ and model parameters
$\lambda, \beta$ therein as follows.
We start out with a ground truth image $g_0 \in \R^{\Omega_s}$ and compute
artificially corrupted wavelet data $g = T g_0 := R T^\infty g_0 \in
\R^{\Omega_s}$ using an operator $R$ as follows.
We select a random subset $J \subset \Omega_s$ by choosing every element of
$\Omega_s$ with probability $\frac{1}{2}$ and define for such fixed $J$ the
operator
$R = R_J: \R^{\Omega_s} \to \R^{\Omega_s}$ by
\[
  (Ru)(x) = \begin{cases}
    u(x) & \text{if $x \not\in J$}, \\
    0 & \text{if $x \in J$}.
  \end{cases}
\]
For model parameters we use $\lambda = \Data{global/basic/lambda}$ and $\beta =
\Data{global/basic/beta}$.

\begin{figure}
  \centering
  \begin{subfigure}[t]{0.3\columnwidth}
    \centering
    \includegraphics[width=\textwidth]{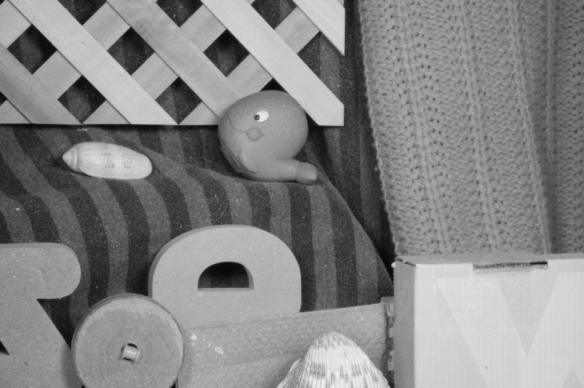}
    \caption{ground truth $g_0$}
  \end{subfigure}
  \begin{subfigure}[t]{0.3\columnwidth}
    \centering
    \includegraphics[width=\textwidth]{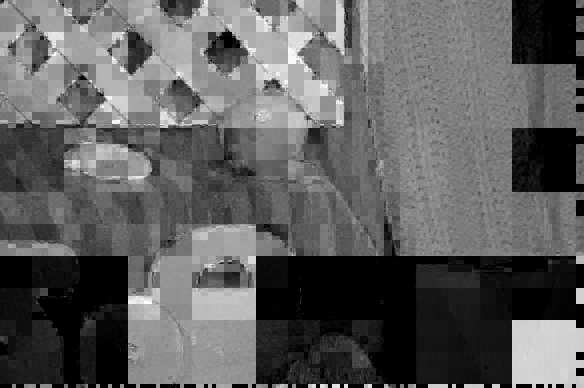}
    \caption{corrupted image $(T^\infty)^{-1} g$}
  \end{subfigure}
  \begin{subfigure}[t]{0.3\columnwidth}
    \centering
    \includegraphics[width=\textwidth]{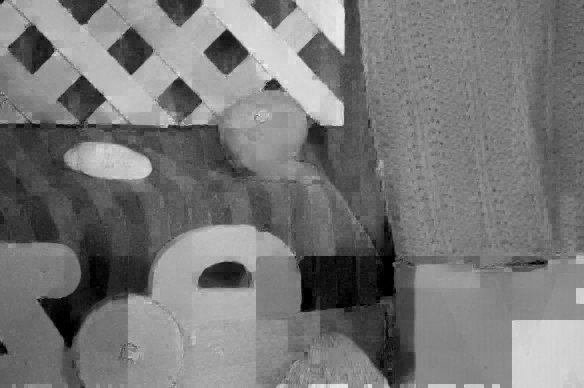}
    \caption{reconstruction $u$}
  \end{subfigure}
  \caption{wavelet inpainting}
  \label{fig:wavelet}
\end{figure}

We decompose the domain into $M = 4 \cdot 4$ domains with an overlap of $r_1 =
r_2 = 5$ pixels and apply
\cref{alg:ddseq} with \cref{alg:ddsur} as a nested subalgorithm.
We use $N_{\text{sur}} = 1$ surrogate iterations,
$\Data{global/basic/ninner}$ iterations for the innermost solver and stop the
outer decomposition algorithm after just $\Data{global/basic/maxiters}$ iterations.
In \cref{fig:wavelet} we can see the used ground truth $g_0$, the corrupted image
visualized as a naive reconstruction $(T^\infty)^{-1} g$ of the corrupted
wavelet data $g$ and the result of
our wavelet inpainting $u$.
Even in regions where bigger chunks of the corrupted image are lost, wavelet
inpainting manages to reconstruct those structures which were preserved by
other wavelet coefficients.

\subsection{Parallel scaling}

\InputData{data1/fd/scaling/opticalflow/data.tex}

\Cref{alg:ddpar,alg:ddseq} allow for a parallel implementation in a domain
decomposition setting.
Indeed, while the subproblems of \cref{alg:ddpar} are independent and may be executed
in parallel without additional consideration, \cref{alg:ddseq} can be parallelized
by applying the algorithm on colored classes of subdomains as in \cite{ChaTaiWanYan} and
calculating the solution on a single colored class of disjoint subdomains in parallel.

We test a parallel implementation of \cref{alg:ddseq} using the same coloring
technique from \cite{ChaTaiWanYan}.
We use $M = \Data{scaling/opticalflow/Mdir} \cdot \Data{scaling/opticalflow/Mdir}$
colored subdomains and otherwise the same parameters and image data as in the
surrogate comparison above.
This means that we have a maximum limit of $9$ disjoint subproblems which can
be scheduled in parallel.
We execute the parallel algorithm with $1, 2, 4$ and $8$ workers respectively
on a Intel(R) Core(TM) i7-5820K CPU @
3.30GHz (6 cores, 12 processing units)
and terminate after reaching an energy of
$\Data{scaling/opticalflow/stopenergy}$.

\begin{figure}[t]
  \centering
  \begin{tikzpicture}
    \begin{loglogaxis}[
      width=0.7\textwidth,
      height=5cm,
      xlabel={number of parallel workers},
      ylabel={runtime (s)},
      legend pos=south west,
      legend cell align=left,
      unbounded coords=jump,
      xtick={1,2,4,8},
      log ticks with fixed point,
      extra x tick style={log identify minor tick positions=false},
      ]
      \addplot table [
          col sep=comma,
          x=nworkers,y=time,
        ] {data1/fd/scaling/opticalflow/timings.csv};
      \addlegendentry{\cref{alg:ddseq}};
      \addplot coordinates {
        (1, 14)
        (8, 10/8)
      };
      \addlegendentry{optimal scaling};
    \end{loglogaxis}
  \end{tikzpicture}
  \caption{time scaling behaviour for the parallel implementation of
  \cref{alg:ddseq} with regard to the number of parallel workers}
  \label{fig:scaling}
\end{figure}
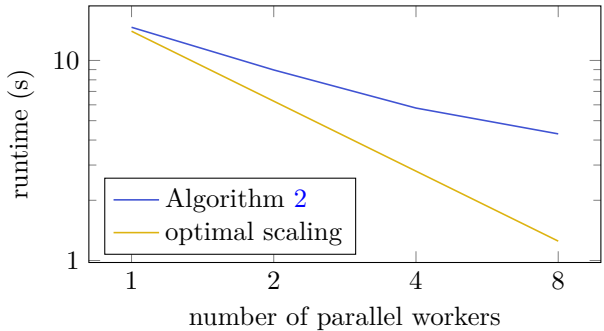

In \cref{fig:scaling} we can see that a parallel implementation can bring about
runtime savings when increasing the number of parallel workers.
The runtime behaves almost inversely linear to the number of workers up to
number of available processor cores, though
the constant factor is not optimal.
We attribute this to the data preparation and communication steps that
are carried out on a single worker and apparently do not scale well in this
implementation.

\section{Conclusion}\label{sec:conclusion}

We have seen that it is possible to improve the domain decomposition
convergence rate results from \cite{ChaTaiWanYan} by making use of different
proof techniques from alternating minimization.
Since as in \cite{ChaTaiWanYan} \cref{alg:ddseq} has a slight
advantage over \cref{alg:ddpar} in terms of iteration count, it suggests that
there is still room for improvement of $\alpha$ in \cref{thm:ddrate} in the
sequential case.

We could easily apply \cref{alg:ddpar,alg:ddseq} to a wider range of local
image processing tasks, namely inpainting and optical flow estimation, while
global operators could only be decomposed by means of the surrogate technique,
which incurred an additional cost.
When considering the total number of iterations of the inner subalgorithm,
\cref{alg:ddpar,alg:ddseq} did not differ substantially in terms of convergence
speed from the global one, i.e. not decomposing the problem, which suggests a
minor overhead of the decomposition method.
A runtime improvement is only to be expected by parallel execution of the
subproblems which we managed to verify in a parallel implementation.
Using the decomposition methods in a memory-constrained computing environment
is expected to be possible.

\section*{Acknowledgement}

This work was partly funded by the Ministerium f\"ur Wissenschaft, Forschung und Kunst of Baden-W\"urttemberg (Az: 7533.-30-10/56/1) through the RISC-project ``Automatische Erkennung von bewegten Objekten in hochaufl\"osenden Bildsequenzen mittels neuer Gebietszerlegungsverfahren'' and by Deutsche Forschungsgemeinschaft (DFG, German Research Foundation) under Germany's Excellence Strategy -- EXC 2075 -- 390740016.


\bibliographystyle{abbrv}

\bibliography{Ref}

\end{document}